\documentclass[12pt]{amsart}
\usepackage{geometry,hyperref}  
\usepackage{graphicx,color}
\usepackage{amssymb}
\usepackage{epstopdf}
\newtheorem{thm}{Theorem}[section]
\newtheorem{conj}[thm]{Conjecture}
\newtheorem{lem}[thm]{Lemma}
\newtheorem{rem}[thm]{Remark}
\newtheorem{open}[thm]{Open question}
\newtheorem{cor}[thm]{Corollary}
\newcommand{\be}{\begin{equation}}
\newcommand{\ee}{\end{equation}}
\newcommand{\baa}{\begin{array}}
\newcommand{\eaa}{\end{array}}
\newcommand{\R}{{\mathbb R}}
\newcommand{\N}{{\mathbb N}}

\renewcommand{\tilde}{\widetilde}
\renewcommand{\epsilon}{\varepsilon}

\begin{document}

\title[Fourth-order Allen-Cahn equation in $\R^N$]{One-dimensional symmetry and Liouville type results for the fourth order Allen-Cahn equation in $\mathbb{R}^{N}$}
\thanks{This work has been carried out in the framework of the Labex Archim\`ede (ANR-11-LABX-0033) and of the A*MIDEX project (ANR-11-IDEX-0001-02), funded by the ``Investissements d'Avenir" French Government program managed by the French National Research Agency (ANR). The research leading to these results has also received funding from the European Research Council under the European Union's Seventh Framework Programme (FP/2007-2013) ERC Grant Agreement n.~321186~- ReaDi~- Reaction-Diffusion Equations, Propagation and Modelling and from the ANR NONLOCAL project (ANR-14-CE25-0013). Part of this work was carried out during visits by D. Bonheure and F.~Hamel to Aix-Marseille University and to the Universit\'e Libre de Bruxelles, whose hospitality is thankfully acknowledged. D.~Bonheure is partially supported by INRIA - Team MEPHYSTO, MIS F.4508.14 (FNRS), PDR T.1110.14F (FNRS) \& ARC AUWB-2012-12/17-ULB1- IAPAS}
\author[D. Bonheure]{Denis Bonheure}
\address{Denis Bonheure, D{\'e}partement de Math{\'e}matique, Universit{\'e} libre de Bruxelles,
\newline \indent CP 214,  Boulevard du Triomphe, B-1050 Bruxelles, Belgium
\newline \indent and INRIA - Team MEPHYSTO.}
\email{denis.bonheure@ulb.ac.be}

\author[F. Hamel]{Fran{\c{c}}ois Hamel}
\address{Fran{\c{c}}ois Hamel, Aix Marseille Universit\'e, CNRS, Centrale Marseille,
\newline \indent  Institut de Math\'ematiques de Marseille, UMR 7373, 13453 Marseille, France}
\email{francois.hamel@univ-amu.fr}

\begin{abstract}
In this paper, we prove an analogue of Gibbons' conjecture for the extended fourth order Allen-Cahn equation in $\R^N$, as well as Liouville type results for some solutions converging to the same value at infinity in a given direction. We also prove a priori bounds and further one-dimensional symmetry and rigidity results for semilinear fourth order elliptic equations with more general nonlinearities.
\end{abstract}

\maketitle

\begin{center}
{\it Dedicated to Ha{\"\i}m Brezis with deep admiration}
\end{center}
\medbreak


\section{Introduction and main results}

We consider the equation
\begin{equation}\label{fourth-original}
\gamma \Delta^2 u-\Delta u=u-u^{3} \ \ \ \text{in}\ \ \mathbb{R}^N
\end{equation}
with $\gamma>0$, which is a fourth order model arising in many bistable physical, chemical or biological systems.
When $\gamma=0$ in~(\ref{fourth-original}), we recognize the well known Allen-Cahn equation
\begin{equation}\label{second}
-\Delta u=u-u^{3} \ \ \ \text{in}\ \ \mathbb{R}^N,
\end{equation}
also called the scalar Ginzburg-Landau equation or the FitzHugh-Nagumo equation. When the solution $u$ is positive and the right-hand side is of the type $u-u^2$, this equation is also known as the Fisher or Kolmogorov-Petrovski-Piskunov equation, originally introduced in~1937~\cite{fi,kpp} as a model for studying biological populations. 

After scaling, the equation~\eqref{fourth-original} with $\gamma>0$ can be written as
\begin{equation}\label{fourth}
\Delta^2 u-\beta\Delta u=u-u^{3} \ \ \ \text{in}\ \ \mathbb{R}^N
\end{equation}
with $\beta=1/\sqrt{\gamma}$. Equation~\eqref{fourth} with $\beta>0$ is usually referred to as the extended Fisher-Kolmogorov equation. It was proposed in 1988 by Dee and van Saarloos~\cite{ds} as a higher order model equation for physical systems that are bistable. The term bistable indicates that the equation~\eqref{second} and its extended version~\eqref{fourth} have two uniform stable states $u(x)=\pm 1$ separated by a third uniform state $u(x)=0$ which is unstable. When $\beta$ is negative, equation~(\ref{fourth}) is related to the stationary version of the Swift-Hohenberg equation
\begin{equation*}
\frac{\partial u}{\partial t} - \kappa u + (1+\Delta)^2 u + u^3 = 0
\end{equation*}
where $\kappa\in\R$. This equation was proposed by Swift and Hohenberg~\cite{sh} as a model in the study of Rayleigh-B\'enard convection. When $\kappa>1$, this equation can be transformed after multiplication and scaling into
$$\frac{\partial u}{\partial t} + (\kappa-1)^{3/2}[ \Delta^{2}u - \beta \Delta u+u^3-u] = 0 $$ with $\beta=-2/\sqrt{\kappa-1}<0.$

For these model equations, a question of great interest is the existence of time-independent phase transitions, i.e. solutions that connect two uniform states in one spatial direction. The second order equation~\eqref{second} has been the subject of a tremendous amount of publications in the past 30 years. Its popularity certainly comes in part from several challenging conjectures raised by De Giorgi~\cite{dg} in 1978. The most famous one, still not completely solved, is:

\begin{conj}[De Giorgi's conjecture]\label{conj:DG}
Suppose that $u$ is a bounded entire solution of~\eqref{second} such that $u_{x_{N}}(x)>0$ for every $x\in\R^{N}$. Then the level sets of $u$ are hyperplanes, at least in dimension $N\le 8$.
\end{conj}

De Giorgi's conjecture has been proved in dimension $N=2$, see \cite[Theorem 1.9]{bcn} and~\cite[Theorem 1.1]{gg}  and in dimension $N=3$, see \cite[Theorem 1.1]{ac}. The conjecture is still open when $4\le N \le 8$, though a positive answer was given in~\cite{s} under the additional assumption
\begin{equation}
\label{limpm1}
\lim_{x_N \to \pm\infty} u(x_1, \dots, x_{N}) = \pm 1
\quad \hbox{for every $x'=(x_1, \dots, x_{N-1} ) \in \R^{N-1}$}
\end{equation}
(see also~\cite{fv2} for more general conditions), while in dimension $N\ge9$ a counterexample has been established in~\cite{dpkw}. We refer to the survey~\cite{fv1} for more details.

On the contrary, up to our knowledge, the fourth order extension~\eqref{fourth} of the Allen-Cahn equation has been only investigated in one spatial dimension, see for instance~\cite{pt4}. In dimension $N>1$, we are only aware of~\cite{Fonseca:2000,hps} where $\Gamma$-limits of scaled energy functionals associated to~(\ref{fourth}) were investigated, and~\cite{bfs} where the authors look for qualitative properties of solutions in a bounded domain under Navier boundary conditions.

The aim of our present work is to make a first study of bounded solutions of~(\ref{fourth}) in any space dimension and in particular to establish one-dimensional symmetry and related Liouville type results.

First of all, we state the analogue of De Giorgi's conjecture for the fourth order Allen-Cahn equation~\eqref{fourth}.

\begin{conj}[Analogue of De Giorgi's conjecture for the fourth order Allen-Cahn equation~\eqref{fourth}]\label{conj:DG-EFK}
Suppose that $\beta\ge \sqrt 8$ and that $u$ is a bounded entire solution of~\eqref{fourth} such that $u_{x_{N}}(x)>0$ for every $x\in\R^{N}$. Then the level sets of~$u$ are hyperplanes, at least in dimension $N\le 8$. 
\end{conj}

This is the expected extension for the fourth order Allen-Cahn equation~\eqref{fourth} of the famous De Giorgi's conjecture for the second order Allen-Cahn equation~\eqref{second}. Indeed, as for~\eqref{second}, the conjecture for~\eqref{fourth} is supported by a $\Gamma$-limit analysis~\cite{Fonseca:2000,hps} and the convergence to the area functional (and the Bernstein problem). Therefore we expect the conjecture to hold true only for $N\le 8$ as for the Allen-Cahn equation. Regarding the condition $\beta\ge\sqrt{8}$ (which means $0<\gamma\le 1/8$ in~\eqref{fourth-original}), it comes from the fact that monotone one-dimensional solutions do not exist for $0\le\beta<\sqrt{8}$, see for instance~\cite{pt4} and Section~\ref{sec1d} below. In fact, one can even wonder whether there exist non-planar standing fronts, i.e. satisfying~\eqref{limpm1}, when $0\le\beta<\sqrt 8$, due to the rich dynamics of bounded solutions in this range of the parameter. 

\begin{open}[Existence of non planar standing fronts for the fourth order Allen\-Cahn equation~\eqref{fourth}]{\rm If $N\ge 2$ and $0\le \beta <\sqrt 8$, are there any solutions of~\eqref{fourth} satis\-fying~\eqref{limpm1} and whose level sets are not hyperplanes~?}
\end{open}

Motivated by the developments on Conjecture~\ref{conj:DG}, we propose to study Conjecture~\ref{conj:DG-EFK} in connection with the additional assumption~\eqref{limpm1} when the limits are uniform in~$(x_1, \dots, x_{N-1} ) \in \R^{N-1}$, namely
\begin{equation}\label{uniform}
\lim_{x_N\to\pm\infty} u(x_1,\dots, x_{N}) = \pm 1
\quad \hbox{uniformly in  $(x_1, \dots, x_{N-1} ) \in \R^{N-1}$.}
\end{equation}
For the second order Allen-Cahn equation~\eqref{second}, under the assumption~\eqref{uniform}, De Giorgi's conjecture is known as Gibbons' conjecture~\cite{g}. In this case, the restriction on the dimension and the monotonicity assumption $\partial_{x_N}u(x)>0$ are unnecessary and Gibbons' conjecture is true for any $N\ge 1$: indeed, it was proved in~\cite{bbg,bhm,f1} that, for any $N\ge 1$, any bounded solution $u$ of~\eqref{second} satisfying~\eqref{uniform} is one-dimensional and thus only depends on the variable $x_N$, that is $u(x)=u(x_N)=\tanh(x_N/\sqrt{2}+a)$ for some $a\in\R$. Furthermore,~$u$ is monotone increasing in $x_N$. We also refer to~\cite{f2,fv3} for further results with discontinuous nonlinearities and in a more abstract setting covering nonlocal operators and fully nonlinear equations.

In our first main result, we prove a result related to Gibbons' conjecture for the fourth order equation~\eqref{fourth}. Again the dimension does not play any role here, but an  additional quantitative bound on $u$ is assumed.

\begin{thm}\label{th1}
For any integer $N\ge 1$ and any real number $\beta\ge\sqrt{8}$, if $u$ is a classical bounded solution of~\eqref{fourth} in $\R^N$ satisfying the uniform limits~\eqref{uniform} and~$\|u\|_{L^{\infty}(\R^N)}<\sqrt{5}$, then $u$ only depends on the variable $x_N$ and is increasing in $x_N$.
\end{thm}

We expect that the assumption $\|u\|_{L^{\infty}(\R^N)}<\sqrt 5$ can be removed though we have to leave that question as open for the moment (see the open question~\ref{openbounds} in Section~\ref{sec3} below). This bound can be slightly relaxed for $\beta > \sqrt 8$ and replaced by a bound depending on~$\beta$ (see Corollary~\ref{cor44} in Section~\ref{sec4} below). In fact, as $\beta\to+\infty$, this quantitative a priori bound becomes somehow just qualitative. Furthermore, we prove in Section~\ref{sec4} a more general version of Theorem~\ref{th1} with a more general right-hand side $f(u)$ instead of $u-u^3$ in~\eqref{fourth}, as well as some further results for solutions which are a priori assumed to range in one side of $\pm 1$ (see Corollaries~\ref{cor43} and~\ref{cor44} below).

On the other hand, the assumption~$\beta\ge\sqrt{8}$ is necessary for the conclusion of Theorem~\ref{th1} to hold, since, even in dimension $N=1$, no bounded monotone solutions of~\eqref{fourth} exist if~$0\le\beta<\sqrt{8}$, see~\cite{pt4}. However one can still ask whether a weaker statement than Theorem \ref{th1} hold or not when~$0\le\beta<\sqrt{8}$, namely if its content holds true without the conclusion regarding the monotonicity in the direction $x_{N}$.

Our second main result deals with a Liouville type property for solutions which are uniformly asymptotic to the same equilibria in a given direction.

\begin{thm}\label{th2}
For any integer $N\ge 1$ and any real number $\beta\ge\sqrt{8}$, if $u$ is a classical bounded solution of~\eqref{fourth} in $\R^N$ satisfying
\begin{equation}\label{uniform2}\left\{\begin{array}{ll}
\displaystyle\!\!\!\lim_{x_N\to\pm\infty}\!u(x_1,\dots,x_N)\!=\!-1\hbox{ uniformly in }(x_1,\dots,x_{N-1})\!\in\!\R^{N-1} & \!\!\!\!\displaystyle\hbox{and }\sup_{\R^N}u<1,\\
\hbox{or} & \\
\displaystyle\!\!\!\lim_{x_N\to\pm\infty}\!u(x_1, \dots, x_{N})\!=\!1\hbox{ uniformly in }(x_1,\dots,x_{N-1})\!\in\!\R^{N-1} & \!\!\!\!\displaystyle\hbox{and }\inf_{\R^N}u>-1,\end{array}\right.
\end{equation}
then $u$ is constant, that is, $u=1$ in $\R^N$ or $u=-1$ in $\R^N$
\end{thm}

In Theorem~\ref{th2}, in contrast to Theorem~\ref{th1}, there is no quantitative a priori bound on the $L^{\infty}(\R^N)$ norm of $u$. Indeed, the condition~\eqref{uniform2} contains in particular the fact that $u$ is a priori assumed to be on one side of $-1$ or $1$, and such bounded solutions then automatically range in $[-1,1]$ (see Corollary~\ref{corbounds} below, and Corollaries~\ref{coroneside} and~\ref{cor43} for further results). On the other hand, as for Theorem~\ref{th1}, the condition $\beta\ge\sqrt{8}$ is necessary for the conclusion to hold, since for $0\le\beta<\sqrt{8}$, even in dimension $N=1$, there are bounded solutions of~\eqref{fourth} and satisfying~\eqref{uniform2} which are not constant, see e.g.~\cite{b,kv,kkv,pt2}. Notice lastly that, thanks to the interior elliptic estimates (see Section~\ref{sec3}), if $u$ is a classical bounded solution of~\eqref{fourth} in $\R^N$ such that $u-1\in L^p(\R^N)$ and $\inf_{\R^N}u>-1$ (resp. $u+1\in L^p(\R^N)$ and~$\sup_{\R^N}u<1$) for some $p\in[1,+\infty)$, then $u(x)\to 1$ (resp. $-1$) as $|x|\to+\infty$, whence~\eqref{uniform2} is automatically fulfilled. As a consequence of Theorem~\ref{th2}, it follows that if $\beta\ge\sqrt{8}$ and if $u$ is a classical bounded solution of~\eqref{fourth} with $u\pm 1\in L^p(\R^N)$ for some~$p\in[1,+\infty)$ and $\inf_{\R^N}|u\mp1|>0$, then $u$ is constant, i.e. $u=\mp 1$ in $\R^N$.

For the second order Allen-Cahn equation~\eqref{second}, any bounded solution $u$ satisfies automatically $\|u\|_{L^{\infty}(\R^N)}\le1$ and either $u=\pm1$ in $\R^N$ or $-1<u<1$ in~$\R^N$, as an imme\-diate consequence of the strong maximum principle. Furthermore, if $u(x_1,\dots,x_N)\to1$ (resp.~$-1$) as $x_N\to\pm\infty$ uniformly in $(x_1,\dots,x_{N-1})\in\R^{N-1}$, then $\inf_{\R^N}u>-1$ (resp.~$\sup_{\R^N}u<1$). By comparing $u$ with shifts of the one-dimensional profile $\tanh(x_N/\sqrt{2})$, it follows that $u$ is constant equal to $1$ (resp. $-1$). This Liouville-type result, which in the one-dimensional case follows from e.g.~\cite{aw,fm}, is actually a special case of more general results demonstrated in~\cite{f4}. It can also be proved with the same method --in the simpler second order case-- as Theorem~\ref{th2} of the present paper and provides an alternative proof to~\cite[Theorem 2]{carbou}. Further Liouville type results for the solutions $u$ of Ginzburg-Landau systems can be found in~\cite{f0,f3}.  

As opposed to the second order Allen-Cahn equation~\eqref{second}, there are bounded solutions~$u$ of the fourth order Allen-Cahn equation~\eqref{fourth} satisfying $\|u\|_{L^{\infty}(\R^N)}>1$, for $0\le\beta<\sqrt{8}$, see e.g. \cite[Chapters 4 \& 5]{pt4}. For $\beta\ge\sqrt{8}$, the bound $\|u\|_{L^{\infty}(\R^N)}\le1$ is expected. But even if $\beta\ge\sqrt{8}$ and $u$ is a priori assumed to satisfy $\|u\|_{L^{\infty}(\R^N)}\le1$ (the bound $\|u\|_{L^{\infty}(\R^N)}<\sqrt{5}$ is actually suffi\-cient to have~$\|u\|_{L^{\infty}(\R^N)}\le1$), it is still an open question to know whether the limits $u(x_1,\dots,x_N)\to1$ (resp. $-1$) as $x_N\to\pm\infty$ uniformly in $(x_1,\dots,x_{N-1})\in\R^{N-1}$ imply $\inf_{\R^N}u>-1$ (resp. $\sup_{\R^N}u<1$). This is why these properties are a priori assumed simultaneously in~\eqref{uniform2}.

\begin{rem}\label{remclassical}{\rm
In Theorems~\ref{th1} and~\ref{th2}, we point out that if $u\in L^{\infty}(\R^N)$ solves~\eqref{fourth} in the sense of distributions, then it is automatically of class $C^{\infty}$ (and in particular, it is a classical solution), see the beginning of Section~\ref{sec3} for more details.}
\end{rem}

\noindent{\bf{Outline.}} The rest of the paper is organized as follows. We first review in Section~\ref{sec1d} some results in the one-dimensional case $N=1$. Section~\ref{sec3} is concerned with the proof of some a priori bounds, for fourth order equations with a right-hand side $f(u)$ more general than~$u-u^3$, when $\beta>0$ is either any positive real number or a large enough real number. These bounds are obtained by writing the fourth order equation~\eqref{fourth} as a system of two second order elliptic equations, see~\eqref{defv}-\eqref{eqv} below. We point out that this system does not satisfy the maximum principle in general. Only the structure of the equation will be used and specific arguments will be developed to establish some a priori bounds. Lastly, Theorems~\ref{th1} and~\ref{th2} are proved in Section~\ref{sec4}. There, we will use a sliding method, inspired by second order equations. More precisely, one uses a sliding method for both $u$ and a function involving $\Delta u$ and one proves that both functions are simultaneously strictly monotone in any direction which is not orthogonal to $x_N$. To do so, for $\beta>0$ large enough, the fourth order equation~\eqref{fourth} is decomposed as a system of two second order equations for which we prove some weak comparison principles in half spaces (see Lemmas~\ref{lem3} and~\ref{lem4} below).


\section{A quick review of the $1D$ case}\label{sec1d}

The study of the equation 
\begin{equation}\label{eq:sta}
u'''' - \beta u'' = u - u^3
\end{equation}
for positive values of the parameter $\beta$ goes back at least to Peletier and Troy in~\cite{pt1,pt2} where they proved, among other things, the existence of kinks for all $\beta>0$. Van den Berg~\cite{vdb2} proved that, when $\beta\ge\sqrt{8}$, the bounded solutions of~(\ref{eq:sta}) behave like the bounded solutions of the stationary Allen-Cahn equation
$$-u''=u-u^3$$
This implies that there exist two kinks (up to translations), one monotone increasing from~$-1$ to~$+1$ and its symmetric, while there are no pulses. When $0\le\beta<\sqrt{8}$, the set of kinks and pulses is much more rich. For this range of $\beta$, kinks and pulses cannot be monotone anymore as at $\beta=\sqrt{8}$ both equilibria $\pm 1$ bifurcate from saddle-nodes to saddle-foci. The linearization of~$(\ref{eq:sta})$ around the equilibria then shows that the solutions oscillate when they are close to $\pm 1$ with small derivatives up to the third order. As $\beta$ becomes smaller than~$\sqrt{8}$, infinitely many kinks and pulses appear. Peletier and Troy~\cite{pt2} proved the existence of two infinite sequences of both kinks and pulses. The two sequences of kinks consist of odd kinks having $2n+1$ zeros and differ in the amplitude of the oscillations. The pulses are even with $2n$ zeros. Again, the two sequences can be distinguished according to the amplitude of the oscillations. Other families of kinks and pulses were shown to exist~\cite{b,kkv,kv}. Basically, these solutions can be distinguished by the number of jumps from~$-1$ to~$+1$ and the oscillations around these equilibria in between the jumps. The complex structure of these solutions can be quantified by defining homotopy classes, see~\cite{kkv}.

Different methods have been used to deal with equation~(\ref{eq:sta}). Peletier and Troy introduced in~\cite{pt1,pt2} a topological shooting method that can be used to track kinks and pulses as well as periodic solutions. In~\cite{pvt}, it is shown that kinks and periodic solutions can be obtained using variational arguments. For instance, if $\beta\ge 0$, the functional
\begin{equation*}
J_{\beta}(u)=\int_{-\infty}^{+\infty}\Big[\frac{1}{2}\big((u''{}^2)+\beta u'{}^2\big)+\frac{1}{4}(u^2-1)^2\Big]\,dx
\end{equation*}
has a minimum in the function space $X=\chi + H^2(\R)$ where $\chi$ is a $C^\infty$ function that satisfies $\chi(x)=-1$ for $x\le -1$ and $\chi(x)=1$ for $x\ge 1$. When $\beta\ge \sqrt{8}$, this minimizer is the unique heteroclinic connection from $-1$ to $+1$, while for $\beta < \sqrt{8}$ it is called the principal heteroclinic as it only has one zero.

The dynamics of equation~(\ref{eq:sta}) with $\beta<0$ is much less understood. Numerical experiments~\cite{vdb1} suggest that a large variety of the solutions found for $\beta$ positive still exist for a certain range of negative values of $\beta$.

In the study of ternary mixtures containing oil, water and amphiphile, a modification of a Ginzburg-Landau model yields for the free energy a functional of the form (see~\cite{gs})
$${\Phi}(u)=\int_{\R^N}[c(\nabla^2u)^2+g(u)\vert\nabla u\vert^2+f(u)]\,dx\,dy\,dz,$$
where the scalar order parameter $u$ is related to the local difference of concentrations of water and oil. The function $g(u)$ quantifies the amphiphilic properties and the ``potential''~$f(u)$ is the bulk free energy of the ternary mixture. In some relevant situations $g$ may take negative values to an extent that is balanced by the positivity of $c$ and $f$.

The admissible density profiles may therefore be identified with critical points of ${\Phi}$ in a suitable function space. In the simplest case where the order parameter depends only on one spatial direction, $u=u(x)$ is defined on the real line and (after scaling) the functional becomes\begin{equation}
\label{eq:functional} {F}(u)=
\int_{-\infty}^{+\infty}\Big[\frac{1}{2}\big(u''{}^2+g(u)u'{}^2\big)+f(u)\Big]\,dx
\end{equation}
whose Euler-Lagrange equation is given by
\begin{equation*}
u''''-g(u)u''-\frac{g'(u)u'{}^2}{2}+f'(u)=0.
\end{equation*}
This model has been considered in~\cite{b,bhs,bstt}. It appears also as a simplification of a nonlocal model due to Andelman et al.~Ê\cite{Kawakatsu:1993}. We refer to Leizarowitz and Mizel~\cite{Leizarowitz:1989}, Coleman, Marcus, and Mizel \cite{coleman:1992} and Mizel, Peletier and Troy~\cite{Mizel1998}. For studies in higher dimension, we refer to Fonseca and Mantegazza~\cite{Fonseca:2000}, Chermisi et al.~\cite{Chermisi:2011} and Hilhorst et al.~\cite{hps}. In the the last quoted paper, the Hessian $\nabla^{2} u$ is replaced by $\Delta u$ as a simplification of the model. This second order energy functional with the Hessian matrix of $u$ replaced by $\Delta u$ was also proposed as model for phase-field theory of edges in anisotropic crystals by Wheeler~\cite{Wheeler:2006}. Finally, we also mention the study of amphiphilic films in~\cite{leibler:1987}.

Functionals of the form~(\ref{eq:functional}) were considered with either a double-well or a triple-well potential $f$ and a function $g$ that can change sign. The case of a triple-well is especially relevant in the theory of ternary mixtures. We refer to~\cite{bbook,bs,pt4} for further references.


\section{A priori bounds}\label{sec3}

In this section, we consider fourth order equations of the type~\eqref{fourth} with any positive real number $\beta>0$ and with a more general right-hand side. Namely, we consider the equation
\begin{equation}\label{fourth2}
\Delta^2 u-\beta\Delta u=f(u) \ \ \ \text{in}\ \ \mathbb{R}^N,
\end{equation}
where $f:\R\to\R$ is any locally Lipschitz-continuous function. We first notice that, from standard elliptic interior estimates~\cite{bjs,dn}, any bounded solution $u$ of~\eqref{fourth2} in the sense of distributions is actually of class $C^{4,\alpha}(\R^N)$, for any $\alpha\in[0,1)$. Therefore, bounded solutions are nothing but classical bounded solutions. By iteration, it follows that, if $f$ is $C^{\infty}(\R)$, then the bounded solutions $u$ are of class $C^{\infty}(\R^N)$ with bounded derivatives at any order. In particular, any bounded solution $u$ of the fourth order Allen-Cahn equation~\eqref{fourth} in the sense of distributions is of class $C^{\infty}(\R^N)$ and has bounded derivatives at any order.

Our goal in this section is to get some a priori pointwise bounds for the bounded solutions of~\eqref{fourth2}, in terms of the function $f$ appearing in the right-hand side. The key-step is given in the following lemma. To lighten some subscript expressions, we use repeatedly the notations
$$M_u:=\sup_{\R^N}u\ \hbox{ and }\ m_u:=\inf_{\R^N}u,$$
where $u$ is a given function.

\begin{lem}\label{lem1}
Let $\beta>0$ and let $f:\R\to\R$ be locally Lipschitz-continuous. If $u$ is a bounded solution of~\eqref{fourth2}, then 
$$\sup_{0<\mu\le\beta^2/4}\!\!\Big[\min_{m_u\le s\le M_{u}}\Big(\frac{f(s)}{\mu}+s\Big)\Big]\le\inf_{\R^N}u\le\sup_{\R^N}u\le\inf_{0<\mu\le\beta^2/4}\!\!\Big[\max_{m_{u}\le s\le M_{u}}\Big(\frac{f(s)}{\mu}+s\Big)\Big]\!.$$
\end{lem}

\begin{proof} Remember first that $u$ is actually a classical solution of~\eqref{fourth2} with bounded and H\"older continuous derivatives up to the fourth order. Fix any real number $\mu\in(0,\beta^2/4]$ and let $\lambda\in(0,\beta)$ be such that
$$\lambda\,(\beta-\lambda)=\mu.$$

{\it Step 1: proof of the right inequality.} Define the $C^2$ bounded function
\begin{equation}\label{defv}
v=\Delta u-\lambda\,u.
\end{equation}
The function $v$ solves
\begin{equation}\label{eqv}
\Delta v-(\beta-\lambda)\,v=\Delta^2u-\beta\Delta u+\lambda\,(\beta-\lambda)\,u=f(u)+\mu\,u\ \hbox{ in }\R^N.
\end{equation}\par
Let $(x_n)_{n\in\N}$ be a sequence of points in $\R^N$ such that $u(x_n)\to M_u$ as $n\to+\infty$. Denote
$$u_n(x)=u(x+x_n)\ \hbox{ and }\ v_n(x)=v(x+x_n)$$
for all $n\in\N$ and $x\in\R^N$. It follows from the aforementioned interior estimates and Ascoli-Arzela theorem that, up to extraction of a subsequence, the functions $u_n$ and $v_n$ converge locally uniformly in $\R^N$ to two functions $u_{\infty}\in C^4(\R^N)$ and $v_{\infty}\in C^2(\R^N)$ solving~\eqref{fourth2},~\eqref{defv} and~\eqref{eqv} with $(u_{\infty},v_{\infty})$ instead of $(u,v)$. Furthermore, $u_{\infty}(0)=M_u\ge u_{\infty}(x)$ for all~$x\in\R^N$. Therefore, $\Delta u_{\infty}(0)\le0$ and
$$v_{\infty}(0)=\Delta u_{\infty}(0)-\lambda\,u_{\infty}(0)\le-\lambda\,u_{\infty}(0)=-\lambda\,M_u.$$
On the other hand, $v_{\infty}\ge m_v$ in $\R^N$, whence
\begin{equation}\label{Muv}
m_v\le-\lambda\,M_u.
\end{equation}\par
Similarly, let $(\xi_n)_{n\in\N}$ be a sequence of points in $\R^N$ such that $v(\xi_n)\to m_v$ as $n\to+\infty$ and denote
$$U_n(x)=u(x+\xi_n)\ \hbox{ and }\ V_n(x)=v(x+\xi_n).$$
As in the previous paragraph, up to extraction of a subsequence, the functions $U_n$ and~$V_n$ converge locally uniformly in $\R^N$ to two functions $U_{\infty}\in C^4(\R^N)$ and $V_{\infty}\in C^2(\R^N)$ sol\-ving~\eqref{fourth2},~\eqref{defv} and~\eqref{eqv} with $(U_{\infty},V_{\infty})$ instead of $(u,v)$. Furthermore,
$$V_{\infty}(0)=m_v\le V_{\infty}(x)\ \hbox{ for all }x\in\R^N.$$
Therefore,
$$f(U_{\infty}(0))+\mu\,U_{\infty}(0)=\Delta V_{\infty}(0)-(\beta-\lambda)\,V_{\infty}(0)\ge-(\beta-\lambda)\,V_{\infty}(0)=-(\beta-\lambda)\,m_v.$$
Using the fact that $0<\lambda<\beta$ and $\lambda\,(\beta-\lambda)=\mu$, one infers from~\eqref{Muv} that
$$f(U_{\infty}(0))+\mu\,U_{\infty}(0)\ge(\beta-\lambda)\,\lambda\,M_u=\mu\,M_u.$$
Since $\inf_{\R^N}u\le U_{\infty}\le\sup_{\R^N}u$ in $\R^N$ and $\mu>0$, it follows that
$$\sup_{\R^N}u=M_u\le\max_{m_u\le s\le M_{u}}\Big(\frac{f(s)}{\mu}+s\Big)$$
and since $\mu\in(0,\beta^2/4]$ was arbitrary, the right inequality in the conclusion of Lemma~\ref{lem1} follows.

\medbreak

{\it Step 2: proof of the left inequality.} Define $\tilde{u}=-u$ and $\tilde{v}=-v$ and observe that $\tilde{v}=\Delta\tilde{u}-\lambda\,\tilde{u}$ and $\Delta\tilde{v}-(\beta-\lambda)\,\tilde{v}=g(\tilde{u})+\mu\,\tilde{u}$ in $\R^N$ with $g(s)=-f(-s)$. Step 1 implies that, for every $0<\mu\le\beta^2/4$,
$$-\inf_{\R^N}u=\sup_{\R^N}\tilde{u}\le\max_{m_{\tilde{u}}\le s\le M_{\tilde{u}}}\Big(\frac{g(s)}{\mu}+s\Big)=-\min_{m_u\le s'\le M_u}\Big(\frac{f(s')}{\mu}+s'\Big)$$
by setting $s=-s'$. Since $\mu\in(0,\beta^2/4]$ can be arbitrary, the left inequality in the conclusion of Lemma~\ref{lem1} follows and the proof of Lemma~\ref{lem1} is thereby complete.
\end{proof}

\begin{rem}{\rm For the fourth-order Allen-Cahn equation~\eqref{fourth}, the estimates given in Lemma~\ref{lem1} improve the bounds given in~\cite{pt3,vdb2}, even in dimension $N=1$, if $\sqrt{2}\le\beta\le\sqrt{8}$. It can also be used to improve the bound given in \cite{hps} for a range of the parameter.}
\end{rem}

In the proof of Theorems~\ref{th1} and~\ref{th2}, we shall first show that the considered solutions~$u$ of~\eqref{fourth} range in the interval $[-1,1]$. To do so, we need some conditions on $\beta$, namely,~$\beta$ should not be too small (since there are some one-dimensional solutions of~\eqref{fourth} satis\-fying~\eqref{uniform} and not ranging in $[-1,1]$ when $0\le\beta<\sqrt{8}$). More generally speaking, for the equation~\eqref{fourth2}, we will give some improved pointwise estimates depending on $\beta$. We assume that, as far as the locally Lipschitz-continuous function $f:\R\to\R$ is concerned, there exist two real numbers $\alpha_-<\alpha_+$ such that
\begin{equation}\label{hypf}
f(\alpha_{\pm})=0,\ \ f>0\hbox{ on }(-\infty,\alpha_-),\ \ f<0\hbox{ on }(\alpha_+,+\infty)\ \hbox{ and }\ f\not\equiv 0\hbox{ on }[\alpha_-,\alpha_+].
\end{equation}
There is then a smallest real number $\beta_f>0$ such that
\begin{equation}\label{betaf}
\forall\,\mu\ge\frac{\beta_f^2}{4},\ \forall\,s\in[\alpha_-,\alpha_+],\ \ \alpha_-\le\frac{f(s)}{\mu}+s\le\alpha_+.
\end{equation}
For the function $f(u)=u-u^3$ in~\eqref{fourth}, $\alpha_{\pm}=\pm 1$ and $\beta_f=\sqrt{8}$. But, here, apart from the fact that $f$ is not identically equal to $0$ on the interval $[\alpha_-,\alpha_+]$, we do not assume anything about the behavior of $f$ on this interval and about its number of sign changes. 

The following lemma provides some sufficient conditions depending on $\beta$ and on $u$ for a solution $u$ of~\eqref{fourth2} to range in the interval $[\alpha_-,\alpha_+]$.

\begin{lem}\label{lem2}
Let $f$ be a locally Lipschitz-continuous function satisfying~\eqref{hypf} and let $\beta_f>0$ be as in~\eqref{betaf}. Then there is a nonincreasing map $m:[\beta_f,+\infty)\to[-\infty,\alpha_-)$ and there is a nondecreasing map $M:[\beta_f,+\infty)\to(\alpha_+,+\infty]$ such that, if $\beta\ge\beta_f$ and $u$ is any classical bounded solution of~\eqref{fourth2} with
\begin{equation}\label{mMbeta}
m(\beta)<\inf_{\R^N}u\le\sup_{\R^N}u<M(\beta),
\end{equation}
then
\begin{equation}\label{iteration}
\alpha_-\le u\le\alpha_+\ \hbox{ in }\R^N.
\end{equation}
Furthermore, $m(\beta)\to-\infty$ and $M(\beta)\to+\infty$ as $\beta\to+\infty$.
\end{lem}

\begin{proof}
We begin by defining the functions $m$ and $M$. For any $\beta\ge\beta_f$, we denote
\be\label{defmbeta}
m(\beta)=\sup\Big\{s\in(-\infty,\alpha_-];\ \frac{4\,f(s)}{\beta^2}+s=\alpha_-+\alpha_+-s\Big\}
\ee
and
\be\label{defMbeta}
M(\beta)=\inf\Big\{s\in[\alpha_+,+\infty);\ \frac{4\,f(s)}{\beta^2}+s=\alpha_-+\alpha_+-s\Big\},
\ee
with the usual convention that $\sup\emptyset=-\infty$ and $\inf\emptyset=+\infty$. Notice that $m(\beta)<\alpha_-$ and $\alpha_+<M(\beta)$ since $f$ is continuous with $f(\alpha_{\pm})=0$, and $\alpha_-<\alpha_+$. Furthermore, since~$f>0$ on $(-\infty,\alpha_-)$, it follows that $\beta\mapsto m(\beta)$ is nonincreasing on $[\beta_f,+\infty)$ and even decreasing on the interval, if any, where it is finite. Similarly, since $f<0$ on $(\alpha_+,+\infty)$, the function $\beta\mapsto M(\beta)$ is nondecreasing on $[\beta_f,+\infty)$ and even increasing on the interval, if any, where it is finite. Lastly, since $\alpha_-<\alpha_+$, it is immediate to see that $m(\beta)\to-\infty$ and $M(\beta)\to+\infty$ as $\beta\to+\infty$. \par
Consider now any classical bounded solution $u$ of~\eqref{fourth2} with $\beta\ge\beta_f$, and assume that~$\sup_{\R^N}u>\alpha_+$. Lemma~\ref{lem1} implies that
\begin{equation}\label{infsup}
\min_{m_u\le s\le M_u}\Big(\frac{4\,f(s)}{\beta^2}+s\Big)\le\inf_{\R^N}u\le\sup_{\R^N}u\le\max_{m_u\le s\le M_u}\Big(\frac{4\,f(s)}{\beta^2}+s\Big).
\end{equation}
By~\eqref{hypf}, there holds $f(\alpha_+)=0$ and $4f(s)/\beta^2+s<s$ for all $s\in(\alpha_+,+\infty)$, whence
\begin{equation}\label{max1}
\max_{\alpha_+\le s\le M_u}\Big(\frac{4\,f(s)}{\beta^2}+s\Big)<\sup_{\R^N}u.
\end{equation}
On the other hand,
\begin{equation}\label{max2}
\max_{\alpha_-\le s\le\alpha_+}\Big(\frac{4\,f(s)}{\beta^2}+s\Big)\le\alpha_+<\sup_{\R^N}u
\end{equation}
by~\eqref{betaf} (and $\beta\ge\beta_f$) and our assumption $\alpha_+<\sup_{\R^N}u$. It then follows from~\eqref{infsup},~\eqref{max1} and~\eqref{max2} that $\inf_{\R^N}u<\alpha_-$ and
$$\sup_{\R^N}u\le\max_{m_u\le s\le\alpha_-}\Big(\frac{4\,f(s)}{\beta^2}+s\Big).$$\par
Similarly, if $u$ is any classical bounded solution of~\eqref{fourth2} with $\beta\ge\beta_f$ and $\inf_{\R^N}u<\alpha_-$, then $\sup_{\R^N}u>\alpha_+$ and
$$\min_{\alpha_+\le s\le M_u}\Big(\frac{4\,f(s)}{\beta^2}+s\Big)\le\inf_{\R^N}u.$$\par
Consider now any classical bounded solution $u$ of~\eqref{fourth2} with $\beta\ge\beta_f$ and~\eqref{mMbeta}, that is,
$$m(\beta)<\inf_{\R^N}u\le\sup_{\R^N}u<M(\beta),$$
and assume that the conclusion~\eqref{iteration} does not hold (that is, either $\sup_{\R^N}u>\alpha_+$ or~$\inf_{\R^N}u<\alpha_-$). It then follows from the previous two paragraphs that these last two properties hold simultaneously and that
\begin{equation}\label{infsup2}
\min_{\alpha_+\le s\le M_u}\Big(\frac{4\,f(s)}{\beta^2}+s\Big)\le\inf_{\R^N}u<\alpha_-<\alpha_+<\sup_{\R^N}u\le\max_{m_u\le s\le\alpha_-}\Big(\frac{4\,f(s)}{\beta^2}+s\Big).
\end{equation}
Since $\inf_{\R^N}u>m(\beta)$ by assumption, it follows from the definition of $m(\beta)$ and $f(\alpha_-)=0$ with $\alpha_-<\alpha_+$, that
$$\frac{4\,f(s)}{\beta^2}+s<\alpha_-+\alpha_+-s\ \hbox{ for all }s\in\Big[\inf_{\R^N}u,\alpha_-\Big],$$
whence
\be\label{supinfalpha}
\sup_{\R^N}u\le\max_{m_u\le s\le\alpha_-}\Big(\frac{4\,f(s)}{\beta^2}+s\Big)<\max_{m_u\le s\le\alpha_-}(\alpha_-+\alpha_+-s)=\alpha_-+\alpha_+-\inf_{\R^N}u
\ee
by~\eqref{infsup2}. Similarly, since $\sup_{\R^N}u<M(\beta)$, one infers from the definition of $M(\beta)$ that~$4\,f(s)/\beta^2+s>\alpha_-+\alpha_+-s$ for all $s\in[\alpha_+,\sup_{\R^N}u]$, whence
\be\label{supinfalpha2}
\alpha_-+\alpha_+-\sup_{\R^N}u=\min_{\alpha_+\le s\le M_u}(\alpha_-+\alpha_+-s)<\min_{\alpha_+\le s\le M_u}\Big(\frac{4\,f(s)}{\beta^2}+s\Big)\le\inf_{\R^N}u
\ee
by ~\eqref{infsup2}. The inequalities~\eqref{supinfalpha} and~\eqref{supinfalpha2} being impossible simultaneously, one has reached a contradiction. As a conclusion, the property~\eqref{iteration} holds and the proof of Lemma~\ref{lem2} is thereby complete.
\end{proof}

An interesting lesson of the proof of Lemma~\ref{lem2} is the fact that, when $\beta\ge\beta_f$, the conditions $\sup_{\R^N}u>\alpha_+$ and $\inf_{\R^N}u<\alpha_-$ hold simultaneously, independently of the a priori bounds~\eqref{mMbeta}. In other words, any of the two inequalities in~\eqref{iteration} implies the other one, as stated in the following corollary.

\begin{cor}\label{coroneside}
Let $f$ be a locally Lipschitz-continuous function satisfying~\eqref{hypf} and let~$\beta_f>0$ be as in~\eqref{betaf}. If $\beta\ge\beta_f$ and $u$ is a classical bounded solution of~\eqref{fourth2} such that either $u\le\alpha_+$ in $\R^N$ or $u\ge\alpha_-$ in $\R^N$, then $\alpha_-\le u\le\alpha_+$ in $\R^N$.
\end{cor}

Lemma~\ref{lem2} also provides an improvement of some pointwise bounds of $u$. This improvement is valid even when $\beta=\beta_f$, since $m(\beta_f)<\alpha_-$ and $M(\beta_f)>\alpha_+$. In other words, there is $\epsilon>0$ such that, if $\beta\ge\beta_f$ and $u$ is a classical bounded solution of~\eqref{fourth2} such that~$\alpha_--\epsilon\le u\le\alpha_++\epsilon$ in $\R^N$, then $\alpha_-\le u\le\alpha_+$ in $\R^N$.

However, the conclusion of Lemma~\ref{lem2} does not hold in general when $\beta$ is smaller than~$\beta_f$. Indeed, in dimension $N=1$ with $f(s)=s-s^3$, it follows from a continuity argument combined with \cite[Theorem 5.1.1]{pt4} and \cite[Theorem 4]{vdb2} that when $0<\beta_{}\displaystyle{\mathop{\to}^<}\sqrt{8}=\beta_f$, there is a sequence of bounded solutions $u_{n}$ such that~$-1-1/n<\inf_{\R}u_{n}<-1<1<\sup_{\R}u_{n}<1+1/n$.

On the other hand, when $\beta$ becomes larger and larger, the a priori bounds~\eqref{mMbeta} become less and less restrictive, since $m(\beta)\to-\infty$ and $M(\beta)\to+\infty$ as $\beta\to+\infty$. More precisely, the following corollary holds immediately.

\begin{cor}\label{corlarge}
Let $f:\R\to\R$ be a locally Lipschitz-continuous function satisfying~\eqref{hypf}, let $\beta_f>0$ be as in~\eqref{betaf} and let $A\ge0$ be a given nonnegative real number. There is $\beta_{f,A}\ge\beta_f$ such that, if $\beta\ge\beta_{f,A}$ and $u$ is a classical bounded solution of~\eqref{fourth2} with $\|u\|_{L^{\infty}(\R^N)}\le A$, then $\alpha_-\le u\le\alpha_+$ in $\R^N$.
\end{cor}

\begin{rem}{\rm At the limit when $\gamma=1/\beta^2\to0^+$, the fourth order equation $\gamma\Delta^2v-\Delta v=f(v)$ (obtained from~\eqref{fourth2} with the scaling $u(x)=v(x/\sqrt{\beta})$) converges formally to the second order equation $-\Delta v=f(v)$. For this last equation, under the assumption~\eqref{hypf}, it follows easily from the maximum principle that any bounded solution $v$ satisfies $\alpha_-\le v\le\alpha_+$ in~$\R^N$.}
\end{rem}

Moreover, when $f$ is bounded or more generally when $|f(s)|=O(|s|)$ as $s\to\pm\infty$, then it follows from the definitions~\eqref{defmbeta} and~\eqref{defMbeta} of $m(\beta)$ and $M(\beta)$ that $m(\beta)=-\infty$ and~$M(\beta)=+\infty$ for $\beta$ large enough. In other words, the $L^{\infty}$ constraint on $u$ is only qualitative for large $\beta$, and the following corollary holds.

\begin{cor}\label{corbounded}
Let $f$ be a locally Lipschitz-continuous function satisfying~\eqref{hypf} and such that $|f(s)|=O(|s|)$ as $s\to\pm\infty$, and let $\beta_f>0$ be as in~\eqref{betaf}. Then there is $\beta_{f,\infty}\ge\beta_f$ such that, for any $\beta\ge\beta_{f,\infty}$, any classical bounded solution $u$ of~\eqref{fourth2} satisfies $\alpha_-\le u\le\alpha_+$ in $\R^N$.
\end{cor}

To complete this section, let us finally translate the previous results to the case of the fourth order Allen-Cahn equation~\eqref{fourth}, that is, $f(s)=s-s^3$. With the previous notations, one has $\alpha_{\pm}=\pm 1$, $\beta_f=\sqrt{8}$ and it is immediate to check that
$$-m(\beta)=M(\beta)=\sqrt{1+\frac{\beta^2}{2}}.$$
Therefore, the following corollary holds.

\begin{cor}\label{corbounds}
Assume that $\beta\ge\sqrt{8}$. Any classical bounded solution $u$ of~\eqref{fourth} such that $\|u\|_{L^{\infty}(\R^N)}<\sqrt{1+\beta^2/2}$ satisfies $\|u\|_{L^{\infty}(\R^N)}\le1$. In particular, any classical bounded solution $u$ of~\eqref{fourth} such that $\|u\|_{L^{\infty}(\R^N)}<\sqrt{5}$ satisfies $\|u\|_{L^{\infty}(\R^N)}\le1$. Furthermore, any classical bounded solution $u$ of~\eqref{fourth} such that $u\ge-1$ in $\R^N$ or $u\le 1$ in $\R^N$ satisfies~$\|u\|_{L^{\infty}(\R^N)}\le1$.
\end{cor}

\begin{open}\label{openbounds}{\rm 
In dimension $N=1$, it is known that any classical bounded solution~$u$ of~\eqref{fourth} satisfies $\|u\|_{L^{\infty}(\R)}\le\sqrt{2}$ whatever $\beta>0$, see for instance \cite{pt3}. The crucial ingredient to prove this a priori bound in dimension $1$ is the first integral of energy which is not available in higher dimension. In particular, from Corollary~$\ref{corbounds}$, any classical bounded solution $u$ of~\eqref{fourth} with $N=1$ satisfies $\|u\|_{L^{\infty}(\R)}\le1$ if $\beta\ge\sqrt{8}$. We believe that these a priori estimates hold true in higher dimensions $N\ge 2$, at least when uniformity conditions~\eqref{uniform} are assumed, however we leave them as an open problem.}
\end{open}


\section{Proof of the rigidity results}\label{sec4}

This section is devoted to the proof of the main results, Theorems~\ref{th1} and~\ref{th2}. As in Section~\ref{sec3}, we can actually consider more general fourth order equations of the type~\eqref{fourth2} with a locally Lipschitz-continuous function $f:\R\to\R$ satisfying~\eqref{hypf} and being decreasing in neighborhoods of $\alpha_{\pm}$ (by decreasing, we mean strictly decreasing, see in particular the end of the proof of Lemma~\ref{lem3}, where the strict monotonicity is used). Namely, the main two results of this section are the following theorems.

\begin{thm}\label{th3}
Let $f:[\alpha_-,\alpha_+]\to\R$ be a Lipschitz-continuous function such that $f(\alpha_{\pm})=0$ and $f$ is decreasing in $[\alpha_-,\alpha_-+\delta]$ and in $[\alpha_+-\delta,\alpha_+]$ for some $\delta>0$. Let $\omega>0$ be such that
\be\label{hypomega}
\frac{f(s)-f(s')}{s-s'}+\omega\ge0\ \hbox{ for all }\alpha_-\le s\neq s'\le\alpha_+.
\ee
If $\beta\ge2\sqrt{\omega}$ and $u$ is a classical bounded solution of~\eqref{fourth2} satisfying $\alpha_-\le u\le\alpha_+$ in $\R^N$ and
\be\label{uniformalpha}
\lim_{x_N\to\pm\infty} u(x_1,\dots,x_{N})=\alpha_{\pm}\quad \hbox{uniformly in $(x_1,\dots,x_{N-1})\in\R^{N-1}$,}
\end{equation}
then $u$ only depends on the variable $x_N$ and is increasing in $x_N$.
\end{thm}

Notice that the condition~\eqref{hypomega} is equivalent to $f'(s)+\omega\ge0$ for all $s\in[\alpha_-,\alpha_+]$ as soon as $f$ is differentiable in the interval $[\alpha_-,\alpha_+]$.

\begin{rem}\label{remomega}{\rm 
It is immediate to see that if $f(\alpha_{\pm})=0$, $f\not\equiv0$ in $[\alpha_-,\alpha_+]$ and $\omega$ satisfies~\eqref{hypomega}, then
$$2\sqrt{\omega}\ge\beta_f,$$
where $\beta_f>0$ is defined in~\eqref{betaf}. Therefore, the parameters $\beta$ considered in Theorem~\ref{th3} are always larger than or equal to $\beta_f$.}
\end{rem}

\begin{thm}\label{th4}
Let $f:\R\to\R$ be a locally Lipschitz-continuous function satisfying~\eqref{hypf} and such that $f$ is decreasing in $[\alpha_-,\alpha_-+\delta]$ and in $[\alpha_+-\delta,\alpha_+]$ for some $\delta>0$. Let $\omega>0$ satisfy~\eqref{hypomega}. If $\beta\ge2\sqrt{\omega}$ and $u$ is a classical bounded solution of~\eqref{fourth2} satisfying
\begin{equation}\label{uniformalpha2}\left\{\begin{array}{ll}
\displaystyle\!\!\!\lim_{x_N\to\pm\infty}\!u(x_1,\dots,x_N)\!=\!\alpha_-\,\hbox{uniformly in }(x_1,\dots,x_{N-1})\!\in\!\R^{N-1} & \!\!\!\!\!\displaystyle\hbox{and }\sup_{\R^N}u\!<\!\alpha_+,\\
\hbox{or} & \\
\displaystyle\!\!\!\lim_{x_N\to\pm\infty}\!u(x_1, \dots, x_{N})\!=\!\alpha_+\,\hbox{uniformly in }(x_1,\dots,x_{N-1})\!\in\!\R^{N-1} & \!\!\!\!\!\displaystyle\hbox{and }\inf_{\R^N}u\!>\!\alpha_-,\end{array}\right.
\end{equation}
and if there is a one-dimensional solution $\phi:\R\to[\alpha_-,\alpha_+]$ of~\eqref{fourth2} such that $\phi(\pm\infty)=\alpha_{\pm}$, then $u$ is constant. 
\end{thm}

The condition on the existence of a one-dimensional front $\phi$ connecting $\alpha_{\pm}$ is reasonable, in the sense that it has a natural counterpart for second order equations. Indeed, for the equation~$-\Delta u=f(u)$ with a Lipschitz-continuous function $f:[\alpha_-,\alpha_+]\to\R$ satisfying
\be\label{bistable}
f(\alpha_{\pm})=0,\ \ f'(\alpha_{\pm})<0,\ \ f<0\hbox{ on }(\alpha_-,\theta)\ \hbox{ and }\  f>0\hbox{ on }(\theta,\alpha_+)
\ee
for some $\theta\in(\alpha_-,\alpha_+)$, then the existence of a one-dimensional front $\phi$ connecting $\alpha_{\pm}$ is a necessary and sufficient condition for the conclusion of Theorem~\ref{th4} to hold and it is equivalent to $\int_{\alpha_-}^{\alpha_+}f(s)ds=0$, see e.g.~\cite{aw,fm}. Otherwise, there are non-constant pulse like solutions $u$ satisfying~\eqref{uniformalpha2}. For the fourth order equation~\eqref{fourth2}, a necessary condition for the existence of a one-dimensional front $\phi:\R\to[\alpha_-,\alpha_+]$ such that $\phi(\pm\infty)=\alpha_{\pm}$ is that $\int_{\alpha_-}^{\alpha_+}f(s)ds=0$, after integrating the equation of $\phi$ over $\R$ against $\phi'$ and using the fact that the derivatives of $\phi$ (at least up to the fourth order) converge to $0$ at $\pm\infty$ by the interior elliptic estimates~\cite{bjs,dn}.

\begin{rem}\label{remoneside}{\rm In Theorem~$\ref{th4}$, the assumption~\eqref{uniformalpha2} and Corollary~$\ref{coroneside}$ imply that the solution $u$ ranges automatically in the interval $[\alpha_-,\alpha_+]$. In other words, Theorem~$\ref{th4}$ could have been stated as Theorem~$\ref{th3}$, with a function $f$ defined in $[\alpha_-,\alpha_+]$ and a solution $u$ ranging a priori in the interval $[\alpha_-,\alpha_+]$.}
\end{rem}

Before doing the proof of Theorems~\ref{th3} and~\ref{th4}, let us state some immediate corollaries by combining them with Remarks~\ref{remomega} and~\ref{remoneside} and the results of Section~\ref{sec3} (namely, Lemma~\ref{lem2} and Corollaries~\ref{coroneside},~\ref{corlarge},~\ref{corbounded}).

\begin{cor}\label{cor43}
Let $f:\R\to\R$ be a locally Lipschitz-continuous function satisfying~\eqref{hypf} and being decreasing in $[\alpha_-,\alpha_-+\delta]$ and in $[\alpha_+-\delta,\alpha_+]$ for some $\delta>0$. Let $\omega>0$ satisfy~\eqref{hypomega}.
\begin{enumerate}
\item If $\beta\ge2\sqrt{\omega}$ and $u$ is a classical bounded solution of~\eqref{fourth2} satisfying~\eqref{uniformalpha} and such that either $u\le\alpha_+$ in $\R^N$ or $u\ge\alpha_-$ in $\R^N$, then $u$ only depends on the variable~$x_N$ and is increasing in $x_N$.
\item Let $m(\beta)$ and $M(\beta)$ be as in Lemma~$\ref{lem2}$. If $\beta\ge2\sqrt{\omega}$ and $u$ is a classical bounded solution of~\eqref{fourth2} satisfying $m(\beta)<\inf_{\R^N}u\le\sup_{\R^N}u<M(\beta)$ and~\eqref{uniformalpha}, then $u$ only depends on the variable $x_N$ and is increasing in $x_N$.
\item Let $A\ge0$ be given. There is $\beta_{f,A}\ge2\sqrt{\omega}$ such that, if $\beta\ge\beta_{f,A}$ and $u$ is a classical bounded solution of~\eqref{fourth2} satisfying $\|u\|_{L^{\infty}(\R^N)}\le A$ and~\eqref{uniformalpha}, then $u$ only depends on the variable $x_N$ and is increasing in $x_N$.
\item If $|f(s)|=O(|s|)$ as $s\to\pm\infty$, then there is $\beta_{f,\infty}\ge2\sqrt{\omega}$ such that, if $\beta\ge\beta_{f,\infty}$ and $u$ is a classical bounded solution of~\eqref{fourth2} satisfying~\eqref{uniformalpha}, then $u$ only depends on the variable $x_N$ and is increasing in $x_N$.
\item If $\beta\ge2\sqrt{\omega}$, if $u$ is a classical bounded solution of~\eqref{fourth2} satisfying~\eqref{uniformalpha2} and if there is a one-dimensional solution $\phi:\R\to[\alpha_-,\alpha_+]$ of~\eqref{fourth2} such that $\phi(\pm\infty)=\alpha_{\pm}$, then $u$ is constant.
\end{enumerate}
\end{cor}

In the particular case of the fourth order Allen-Cahn equation~\eqref{fourth} with $f(s)=s-s^3$, one has $\alpha_{\pm}=\pm1$ and the smallest $\omega>0$ satisfying~\eqref{hypomega} is equal to $\omega=2$. Furthermore, the existence of one-dimensional kinks $\phi:\R\to[-1,1]$ solving~\eqref{fourth} with $\phi(\pm\infty)=\pm1$ is guaranteed when $\beta\ge\sqrt{8}$, see e.g.~\cite{pt4,vdb2}. Therefore, the following corollary holds, from which Theorems~\ref{th1} and~\ref{th2} follow.

\begin{cor}\label{cor44}\begin{enumerate}
\item If $\beta\ge\sqrt{8}$ and $u$ is a classical bounded solution of~\eqref{fourth} satisfying $\|u\|_{L^{\infty}(\R^N)}<\sqrt{1+\beta^2/2}$ and~\eqref{uniform}, then $u$ only depends on the variable $x_N$ and is increasing in $x_N$. The same conclusion holds if, instead of $\|u\|_{L^{\infty}(\R^N)}<\sqrt{1+\beta^2/2}$, $u$ is assumed to satisfy either $u\le1$ in $\R^N$ or $u\ge-1$ in $\R^N$. Furthermore, for any $A\ge0$, there is $\tilde{\beta}_A\ge\sqrt{8}$ such that, if $\beta\ge\tilde{\beta}_A$ and $u$ is a classical bounded solution of~\eqref{fourth} satisfying $\|u\|_{L^{\infty}(\R^N)}\le A$ and~\eqref{uniform}, then $u$ only depends on the variable~$x_N$ and is increasing in $x_N$.\par
\item If $\beta\ge\sqrt{8}$ and $u$ is a classical bounded solution of~\eqref{fourth} satisfying~\eqref{uniform2}, then $u$ is constant.
\end{enumerate}
\end{cor}

It now only remains to prove Theorems~\ref{th3} and~\ref{th4}. A key-point is the following result, which can be viewed as a weak maximum principle in half-spaces for problem~\eqref{fourth2} when~$\beta$ is large enough. We denote $\R^N_+=\R^{N-1}\times[0,+\infty)$.

\begin{lem}\label{lem3}
Let $f:[\alpha_-,\alpha_+]\to\R$ be a Lipschitz-continuous function such that $f(\alpha_+)=0$ and $f$ is decreasing in $[\alpha_+-\delta,\alpha_+]$ for some $\delta>0$. Assume $\omega>0$ satisfies~\eqref{hypomega}, $\beta\ge2\sqrt{\omega}$ and $\lambda>0$ is any of the roots of the equation $\lambda^2-\beta\lambda+\omega=0$. If $z_1$ and $z_2$ are two classical bounded solutions of~\eqref{fourth2} such that $\alpha_-\le z_1,z_2\le\alpha_+$ in $\R^N$ and 
\be\label{hypz12}\left\{\begin{array}{l}
z_2\ge\alpha_+-\delta\hbox{ in }\R^N_+,\vspace{3pt}\\
\displaystyle\lim_{x_N\to+\infty}z_1(x_1,\dots,x_N)=\lim_{x_N\to+\infty}z_2(x_1,\dots,x_N)=\alpha_+\vspace{-3pt}\\
\qquad\qquad\qquad\qquad\qquad\qquad\qquad\qquad\hbox{ uniformly in }(x_1,\dots,x_{N-1})\in\R^{N-1},\vspace{3pt}\\
z_1\le z_2\hbox{ and }\Delta z_1-\lambda\,z_1\ge\Delta z_2-\lambda\,z_2\hbox{ on }\partial\R^N_+=\R^{N-1}\times\{0\},\end{array}\right.
\ee
then $z_1\le z_2$ and $\Delta z_1-\lambda\,z_1\ge\Delta z_2-\lambda\,z_2$ in $\R^N_+$.
\end{lem}

\begin{proof}
By writing 
$$v=z_2-z_1\ \hbox{ and }\ w=\Delta v-\lambda\,v,$$
the desired conclusion can be reformulated as $\inf_{\R^N_+}v\ge0$ and $\sup_{\R^N_+}w\le0$. Assume for the sake of contradiction that
\be\label{mM}
m:=\inf_{\R^N_+}v<0\ \hbox{ or }\ M:=\sup_{\R^N_+}w>0.
\ee
Let $(x_n)_{n\in\N}=(x'_n,x_{N,n})_{n\in\N}$ and $(y_n)_{n\in\N}=(y'_n,y_{N,n})_{n\in\N}$ be some sequences in $\R^N_+$ such that
$$v(x_n)\to m\ \hbox{ and }w(y_n)\to M\ \hbox{ as }n\to+\infty.$$
By standard interior elliptic estimates~\cite{bjs,dn}, as recalled at the beginning of Section~\ref{sec3}, the functions $z_1$ and $z_2$ have bounded H\"older continuous derivatives up to the fourth order. In particular, both functions $v$ and $w$ are uniformly continuous. Furthermore, since~$\lim_{x_N\to+\infty}z_1(x',x_N)=\lim_{x_N\to+\infty}z_2(x',x_N)=\alpha_+$ uniformly in $x'\in\R^{N-1}$, one infers that $\lim_{x_N\to+\infty}\Delta z_1(x',x_N)=\lim_{x_N\to+\infty}\Delta z_2(x',x_N)=0$ uniformly in $x'\in\R^{N-1}$. Hence,
$$\lim_{x_N\to+\infty}v(x',x_N)=0\ \hbox{ and }\ \lim_{x_N\to+\infty}w(x',x_N)=0\ \hbox{ uniformly in }x'\in\R^{N-1}.$$

\medbreak

{\it Claim 1: $m<0$ implies $M>0$}. Assume $m<0$. Since $v\ge0$ on $\partial\R^N_+$, it follows from the previous observations that there exists a real number~$x_{N,\infty}\in(0,+\infty)$ such that, up to extraction of a subsequence,~$x_{N,n}\to x_{N,\infty}$ as $n\to+\infty$. Denote
$$Z_{1,n}(x)=z_1(x'+x'_n,x_N)\hbox{ and }Z_{2,n}(x)=z_2(x'+x'_n,x_N)$$
for all $n\in\N$ and $x=(x',x_N)\in\R^N$. Up to extraction of a subsequence, the functions~$Z_{1,n}$ and~$Z_{2,n}$ converge in $C^4_{loc}(\R^N)$ to two classical bounded solutions $Z_1$ and $Z_2$ of~\eqref{fourth2} such that $\alpha_-\le Z_1,Z_2\le\alpha_+$ in $\R^N$, $Z_2\ge\alpha_+-\delta$ in $\R^N_+$, $V:=Z_2-Z_1\ge m$ in $\R^N_+$, $V(0,x_{N,\infty})=m$ and $W:=\Delta V-\lambda\,V\le M$ in $\R^N_+$. In particular, since $(0,x_{N,\infty})$ is an interior minimum point of $V$ in $\R^N_+$, one infers that
\be\label{lambdam}
M\ge W(0,x_{N,\infty})=\Delta V(0,x_{N,\infty})-\lambda\,V(0,x_{N,\infty})\ge-\lambda\,V(0,x_{N,\infty})=-\lambda\,m>0
\ee
since $\lambda>0$, and $m<0$ by assumption. Therefore, $m<0$ implies $M>0$ as claimed.

\medbreak

{\it Claim 2: $M>0$ implies $m<0$}. Assume now that $M>0$. Since $w\le0$ on $\partial\R^N_+$ and~$\lim_{x_N\to+\infty}w(x',x_N)=0$ uniformly in $x'\in\R^{N-1}$, it follows that there exists a real number~$y_{N,\infty}\in(0,+\infty)$ such that, up to extraction of a subsequence, $y_{N,n}\to y_{N,\infty}$ as~$n\to+\infty$. Denote
$$\tilde{Z}_{1,n}(x)=z_1(x'+y'_n,x_N)\hbox{ and }\tilde{Z}_{2,n}(x)=z_2(x'+y'_n,x_N)$$
for all $n\in\N$ and $x=(x',x_N)\in\R^N$. Up to extraction of a subsequence, the functions~$\tilde{Z}_{1,n}$ and~$\tilde{Z}_{2,n}$ converge in $C^4_{loc}(\R^N)$ to two classical bounded solutions $\tilde{Z}_1$ and $\tilde{Z}_2$ of~\eqref{fourth2} such that $\alpha_-\le\tilde{Z}_1,\tilde{Z}_2\le\alpha_+$ in $\R^N$, $\tilde{Z}_2\ge\alpha_+-\delta$ in $\R^N_+$, $\tilde{V}:=\tilde{Z}_2-\tilde{Z}_1\ge m$ in $\R^N_+$, $\tilde{W}:=\Delta\tilde{V}-\lambda\tilde{V}\le M$ in $\R^N_+$ and $\tilde{W}(0,y_{N,\infty})=M$. Let $\tilde{\lambda}>0$ be the other root of the equation
$$\tilde{\lambda}^2-\beta\tilde{\lambda}+\omega=0,$$
that is, $\tilde{\lambda}=\omega/\lambda=\beta-\lambda$. The function $\tilde{W}=\Delta\tilde{V}-\lambda\,\tilde{V}$ is of class $C^2(\R^N)$ and it satisfies
$$\Delta\tilde{W}-\tilde{\lambda}\,\tilde{W}=\Delta^2\tilde{V}-\beta\Delta\tilde{V}+\omega\,\tilde{V}=f(\tilde{Z}_2)-f(\tilde{Z}_1)+\omega\,(\tilde{Z}_2-\tilde{Z}_1)\hbox{ in }\R^N.$$
The point $(0,y_{N,\infty})$ is an interior maximum point of the function $\tilde{W}$ in $\R^N_+$, whence $\Delta\tilde{W}(0,y_{N,\infty})\le0$, while $\tilde{\lambda}>0$ and $\tilde{W}(0,y_{N,\infty})=M>0$. Therefore,
\be\label{lambdaM}\baa{rcl}
0>-\tilde{\lambda}\,M & \!\!\!\ge\!\!\! & \Delta\tilde{W}(0,y_{N,\infty})-\tilde{\lambda}\,\tilde{W}(0,y_{N,\infty})\vspace{3pt}\\
& \!\!\!=\!\!\! & f(\tilde{Z}_2(0,y_{N,\infty}))-f(\tilde{Z}_1(0,y_{N,\infty}))+\omega\,(\tilde{Z}_2(0,y_{N,\infty})-\tilde{Z}_1(0,y_{N,\infty}))\vspace{3pt}\\
& \!\!\!=\!\!\! & (\sigma+\omega)\,\tilde{V}(0,y_{N,\infty}),\eaa
\ee
where
$$\sigma=\frac{f(\tilde{Z}_2(0,y_{N,\infty}))-f(\tilde{Z}_1(0,y_{N,\infty}))}{\tilde{Z}_2(0,y_{N,\infty})-\tilde{Z}_1(0,y_{N,\infty})}\ \hbox{ if }\tilde{V}(0,y_{N,\infty})=\tilde{Z}_2(0,y_{N,\infty})-\tilde{Z}_1(0,y_{N,\infty})\neq0$$
and, say, $\sigma=0$ if $\tilde{V}(0,y_{N,\infty})=\tilde{Z}_2(0,y_{N,\infty})-\tilde{Z}_1(0,y_{N,\infty})=0$. Since $\alpha_-\le\tilde{Z}_1(0,y_{N,\infty}),\tilde{Z}_2(0,y_{N,\infty})\le\alpha_+$ and $\omega>0$ satisfies~\eqref{hypomega}, one has $\sigma+\omega\ge0$. The case~$\sigma+\omega=0$ is impossible due to the strict sign in~\eqref{lambdaM}, whence $\sigma+\omega>0$. As a consequence, $\tilde{V}(0,y_{N,\infty})<0$, that is, $\tilde{Z}_2(0,y_{N,\infty})<\tilde{Z}_1(0,y_{N,\infty})$. In particular,~$m\le\tilde{V}(0,y_{N,\infty})<0$. Therefore our claim follows. Since $\sigma+\omega>0$, observe furthermore that~\eqref{lambdaM} implies
\be\label{tildelambda}
-\tilde{\lambda}\,M\ge(\sigma+\omega)\,m.
\ee

\medbreak

{\it Conclusion}. The previous claims show that the property $m<0$ is equivalent to~$M>0$. Therefore, if we assume~\eqref{mM}, the calculations in the proofs of the two claims hold. Now, by multiplying~\eqref{lambdam} by $\tilde{\lambda}>0$ and adding it to~\eqref{tildelambda}, one gets that
$$0\ge-\lambda\,\tilde{\lambda}\,m+(\sigma+\omega)\,m=\sigma\,m$$
since $\lambda\,\tilde{\lambda}=\omega$. As $m<0$, this yields $\sigma\ge0$, whence
\be\label{finfty}
f(\tilde{Z}_2(0,y_{N,\infty}))\le f(\tilde{Z}_1(0,y_{N,\infty}))
\ee
(remember that $\tilde{Z}_2(0,y_{N,\infty})<\tilde{Z}_1(0,y_{N,\infty})$). On the other hand, $\tilde{Z}_2(0,y_{N,\infty})\ge\alpha_+-\delta$ (since $z_2,\tilde{Z}_2\ge\alpha_+-\delta$ in $\R^N_+$) and $\tilde{Z}_1(0,y_{N,\infty})\le\alpha_+$, whence
$$\alpha_+-\delta\le\tilde{Z}_2(0,y_{N,\infty})<\tilde{Z}_1(0,y_{N,\infty})\le\alpha_+.$$
Since $f$ is assumed to be decreasing in the interval $[\alpha_+-\delta,\alpha_+]$, one gets that $f(\tilde{Z}_2(0,y_{N,\infty}))>f(\tilde{Z}_1(0,y_{N,\infty}))$, contradicting~\eqref{finfty}. Henceforth, we conclude that~\eqref{mM} cannot hold, that is, $\inf_{\R^N}v=m\ge0$, $\sup_{\R^N}w=M\le0$ and the proof is thereby complete.
\end{proof}

\begin{rem}{\rm It follows from the proof that the functions $z_1$ and $z_2$ satisfy the following two-component system of second order elliptic equations:
$$\left\{\baa{rcl}
-\Delta v+\lambda\,v & = & -w,\vspace{3pt}\\
-\Delta w+\tilde{\lambda}\,w & = & -f(z_2)+f(z_1)-\omega\,(z_2-z_1)=-b(x)\,v\eaa\right.$$
in $\R^N_+$, with $v=z_2-z_1$, $b(x)\ge0$ (by~\eqref{hypomega}) in $\R^N_+$, $v\ge0$ on $\partial\R^N_+$ and $w\le0$ on $\partial\R^N_+$. This system is competitive and some maximum principles are known to hold for competitive systems in bounded domains. However, the comparison result proved in Lemma~\ref{lem3} is new. Here, not only the domain $\R^N_+$ is unbounded and one also has to use in the proof the fact that, in $\R^N_+$, $z_2$ takes values in the interval $[\alpha_+-\delta,\alpha_+]$, in which $f$ is decreasing.}
\end{rem}

\begin{rem}{\rm The second line in~\eqref{hypz12} is actually not necessary. To see it, notice first that, since the functions $z_1$ and $z_2$ are assumed to be bounded in $\R^N$ (they range in $[\alpha_-,\alpha_+]$), the function $w$ is bounded in $\R^N$ from the interior elliptic estimates. If we drop the second line in~\eqref{hypz12}, the main change is that any of the sequences $(x_{N,n})_{n\in\N}$ and $(y_{N,n})_{n\in\N}$ introduced in the proof may well converge to $+\infty$ up to extraction of a subsequence. If for instance $\lim_{n\to+\infty}x_{N,n}=+\infty$, one would now define $Z_{i,n}(x)=z_i(x+x_n)=z_i(x'+x'_n,x_N+x_{N,n})$ for $i=1,2$ and pass to the limit up to extraction of a subsequence. The limiting functions $Z_1$, $Z_2$, $V$ and $W$ would then satisfy the same properties as in the proof of Lemma~\ref{lem3}, but now in the whole space $\R^N$ (in particular, $Z_2\ge\alpha_+-\delta$ in $\R^N$, $V=Z_2-Z_1\ge m$ in~$\R^N$, $V(0,0)=m$ and $W=\Delta V-\lambda\,V\le M$ in $\R^N$). The inequalities established in the above proof would still be the same, with $x_{N,\infty}=0$. Similarly, if $\lim_{n\to+\infty}y_{N,n}=+\infty$, one would define new functions $\tilde{Z}_{i,n}$ and new limiting functions $\tilde{Z}_i$, $\tilde{V}$ and $\tilde{W}$ and one would complete the proof with $y_{N,\infty}=0$. In the statement of Lemma~\ref{lem3}, we preferred to keep the second line in~\eqref{hypz12} for the sake of simplicity, and since this assumption will always be satisfied when Lemma~\ref{lem3} will be used in the proof of Theorems~\ref{th3} and~\ref{th4}.}
\end{rem}

The following lemma is the counterpart of Lemma~\ref{lem3} in the half-space $\R^N_-=\R^{N-1}\times(-\infty,0]$, when $z_1$ is close to $\alpha_-$.

\begin{lem}\label{lem4}
Let $f:[\alpha_-,\alpha_+]\to\R$ be a Lipschitz-continuous function such that $f(\alpha_-)=0$ and $f$ is decreasing in $[\alpha_-,\alpha_-+\delta]$ for some $\delta>0$. Let $\omega>0$ satisfy~\eqref{hypomega}, let $\beta\ge2\sqrt{\omega}$ and let $\lambda>0$ be any root of the equation $\lambda^2-\beta\lambda+\omega=0$. If $z_1$ and $z_2$ are two classical bounded solutions of~\eqref{fourth2} such that $\alpha_-\le z_1,z_2\le\alpha_+$ in $\R^N$ and 
$$\left\{\begin{array}{l}
z_1\le\alpha_-+\delta\hbox{ in }\R^N_+,\vspace{3pt}\\
\displaystyle\lim_{x_N\to-\infty}z_1(x_1,\dots,x_N)=\lim_{x_N\to-\infty}z_2(x_1,\dots,x_N)=\alpha_-\\
\qquad\qquad\qquad\qquad\qquad\qquad\qquad\qquad\hbox{ uniformly in }(x_1,\dots,x_{N-1})\in\R^{N-1},\vspace{3pt}\\
z_1\le z_2\hbox{ and }\Delta z_1-\lambda\,z_1\ge\Delta z_2-\lambda\,z_2\hbox{ on }\partial\R^N_-=\R^{N-1}\times\{0\},\end{array}\right.$$
then $z_1\le z_2$ and $\Delta z_1-\lambda\,z_1\ge\Delta z_2-\lambda\,z_2$ in $\R^N_-$.
\end{lem}

\begin{proof}
The conclusion of Lemma~\ref{lem4} follows immediately from Lemma~\ref{lem3} applied to the functions $\tilde{f}(s)=-f(\alpha_-+\alpha_+-s)$, $\tilde{z}_1(x',x_N)=\alpha_-+\alpha_+-z_2(x',-x_N)$ and $\tilde{z}_2(x',x_N)=\alpha_-+\alpha_+-z_1(x',-x_N)$.
\end{proof}

With Lemmas~\ref{lem3} and~\ref{lem4} at hand, we can turn to the proof of Theorems~\ref{th3} and~\ref{th4}. We rely on the sliding method adapted to the fourth order equation~\eqref{fourth2}. Namely, for any vector $\xi'\in\R^{N-1}$, any $\tau\in\R$ and $x=(x',x_N)\in\R^N$, we set
\begin{equation}\label{eq:utau}
u_{\tau}(x)=u_{\tau}(x',x_N)=u(x'+\xi',x_N-\tau).
\end{equation}
The strategy then consists in two main steps. First we show that $u_{\tau}\le u$ in $\R^N$ for $\tau>0$ large enough and then we decrease $\tau$ and prove that actually, $u_{\tau}\le u$ in $\R^N$ for all $\tau\ge0$. Finally we show that the freedom in the choice of the vector $\xi'\in\R^{N-1}$ implies that $u$ depends only on $x_{N}$.\hfill\break

\begin{proof}[Proof of Theorem~$\ref{th3}$.]
Let $f$, $\delta$, $\omega$, $\beta$ and $u$ be as in the statement of Theorem~\ref{th3}. Notice that the strict monotonicity of $f$ in $[\alpha_-,\alpha_-+\delta]$ and $[\alpha_+-\delta,\alpha_+]$ implies that
$$\delta<\frac{\alpha_+-\alpha_-}{2}.$$
Let $\lambda>0$ be any root of the equation $\lambda^2-\beta\,\lambda+\omega=0$.\par
Since $u(x',x_N)$ converges to the constants $\alpha_{\pm}$ as $x_N\to\pm\infty$ and since the function $u$ and its derivatives up to the fourth order are bounded and H\"older continuous in $\R^N$ from the interior elliptic estimates~\cite{bjs,dn}, it follows in particular that $\Delta u(x',x_N)\to0$ as $x_N\to\pm\infty$ uniformly in $x'\in\R^{N-1}$. From the assumption~\eqref{uniformalpha}, there is a real number $A>0$ such that
\be\label{choiceA}\left\{\baa{l}
u\ge\alpha_+-\delta\hbox{ in }\R^{N-1}\times[A,+\infty),\ \ u\le\alpha_-+\delta\hbox{ in }\R^{N-1}\times(-\infty,-A],\vspace{3pt}\\
\displaystyle|\Delta u|\le\frac{\lambda\,(\alpha_+-\alpha_--2\delta)}{2}\hbox{ in }\R^{N-1}\times(-\infty,-A]\ \cup\ \R^{N-1}\times[A,+\infty).\eaa\right.
\ee

\medbreak

{\it Step 1: comparisons for large $\tau$.} We show here that $u_{\tau}\le u$ in $\R^N$ for all $\tau\ge2A$. Let~$\tau\in[2A,+\infty)$ be given. For any $x'\in\R^{N-1}$, one has
\be\label{ineq1}
u_{\tau}(x',A)=u(x'+\xi',A-\tau)\le\alpha_-+\delta\le\alpha_+-\delta\le u(x',A)
\ee
and
\be\label{ineq2}\baa{rcl}
\Delta u_{\tau}(x',A)-\lambda\,u_{\tau}(x',A) & = & \Delta u(x'+\xi',A-\tau)-\lambda\,u(x'+\xi',A-\tau)\vspace{3pt}\\
& \ge & \displaystyle-\frac{\lambda\,(\alpha_+-\alpha_--2\delta)}{2}-\lambda\,(\alpha_-+\delta)\vspace{3pt}\\
& = & \displaystyle\frac{\lambda\,(\alpha_+-\alpha_--2\delta)}{2}-\lambda\,(\alpha_+-\delta)\vspace{3pt}\\
& \ge & \Delta u(x',A)-\lambda\,u(x',A).\eaa
\ee

Define $z_{1},z_{2}:\R^N\to[\alpha_-,\alpha_+]$ by $z_1:=u_{\tau}(\cdot,\cdot+A)$ and $z_2:=u(\cdot,\cdot+A)$. Since $z_{2} \ge\alpha_+-\delta$ in $\R^N_+$ and both $z_{1}$ and $z_{2}$ are classical solutions of~\eqref{fourth2} converging to $\alpha_+$ as $x_N\to+\infty$ uniformly in $x'\in\R^{N-1}$, we deduce from~\eqref{ineq1}-\eqref{ineq2} and Lemma~\ref{lem3} that $z_1\le z_2$ and $\Delta z_1-\lambda\,z_1\ge\Delta z_2-\lambda\,z_2$ in $\R^N_+$, that is,
$$u_{\tau}\le u\ \hbox{ and }\ \Delta u_{\tau}-\lambda\,u_{\tau}\ge\Delta u-\lambda\,u\ \hbox{ in }\R^{N-1}\times[A,+\infty).$$\par
Similarly, since $z_1=u_{\tau}(\cdot,\cdot+A)=u(\cdot+\xi',\cdot+A-\tau)\le\alpha_-+\delta$ in $\R^N_-$ (because $A-\tau\le-A$) and both $z_1$ and $z_2$ converge to $\alpha_-$ as $x_N\to-\infty$ uniformly in $x'\in\R^{N-1}$, it follows from~\eqref{ineq1}-\eqref{ineq2} and Lemma~\ref{lem4} that $z_1\le z_2$ and $\Delta z_1-\lambda\,z_1\ge\Delta z_2-\lambda\,z_2$ in~$\R^N_-$, that is,
$$u_{\tau}\le u\ \hbox{ and }\ \Delta u_{\tau}-\lambda\,u_{\tau}\ge\Delta u-\lambda\,u\ \hbox{ in }\R^{N-1}\times(-\infty,A].$$
Therefore, for all $\tau\ge2A$, $u_{\tau}\le u$ and $\Delta u_{\tau}-\lambda\,u_{\tau}\ge\Delta u-\lambda\,u$ in $\R^N$.\par

\medbreak

{\it Step 2: decreasing $\tau$.} Let us now set
\be\label{deftau*}
\tau^*=\inf\big\{\tau>0,\ u_{\tau'}\le u\hbox{ and }\Delta u_{\tau'}-\lambda\,u_{\tau'}\ge\Delta u-\lambda\,u\hbox{ in }\R^N\hbox{ for all }\tau'\ge\tau\big\}.
\ee
It follows from Step 1 that $\tau^*$ is a nonnegative real number such that $\tau^*\le 2A$. Furthermore, by continuity,
\be\label{utau*}
u_{\tau^*}\le u\ \hbox{ and }\ \Delta u_{\tau^*}-\lambda\,u_{\tau^*}\ge\Delta u-\lambda\,u\ \hbox{ in }\R^N.
\ee
Our final aim is to show that $\tau^*=0$. Assume by contradiction that $\tau^*>0$.\par
We first claim that
\be\label{claim1}
\inf_{\R^{N-1}\times[-A,A]}(u-u_{\tau^*})>0.
\ee
Assume by contradiction that this is not the case. Then, we have  
$$\inf_{\R^{N-1}\times[-A,A]}(u-u_{\tau^*})=0$$ and there exists a sequence $(x_n)_{n\in\N}=(x'_n,x_{N,n})_{n\in\N}$ of points in $\R^{N-1}\times[-A,A]$ such that~$u(x_n)-u_{\tau^*}(x_n)\to0$ as $n\to+\infty$. Up to extraction of a subsequence, one can assume that $x_{N,n}\to x_{N,\infty}\in[-A,A]$ as $n\to+\infty$. Define
$$U_n(x)=U_n(x',x_N)=u(x'+x'_n,x_N)$$
in $\R^N$. The functions $U_n$ are bounded in $C^{4,\alpha}(\R^N)$ for any $\alpha\in[0,1)$ from the interior elliptic estimates~\cite{bjs,dn}. Up to extraction of a subsequence, they converge in $C^4_{loc}(\R^N)$ to a classical solution $U:\R^N\to[\alpha_-,\alpha_+]$ of~\eqref{fourth2} such that
$$V:=U-U_{\tau^*}=U-U(\cdot+\xi',\cdot-\tau^*)\ge 0\ \hbox{ in }\R^N,$$
$V(0,x_{N,\infty})=0$ and $W:=\Delta V-\lambda\,V\le 0$ in $\R^N$. The strong maximum principle applied to $V$ yields $V=0$ in $\R^N$, that is, $U(x',x_N)=U(x'+\xi',x_N-\tau^*)$ for all $x=(x',x_N)\in\R^N$. However, $U$ still satisfies~\eqref{uniformalpha} by the uniformity with respect to $x'=(x_1,\dots,x_{N-1})\in\R^{N-1}$ in~\eqref{uniformalpha}. Thus, the positivity of $\tau^*$ leads to a contradiction. As a consequence,~\eqref{claim1} holds.\par
We next claim that
\be\label{claim2}
\sup_{\R^{N-1}\times[-A,A]}\big(\Delta(u-u_{\tau^*})-\lambda\,(u-u_{\tau^*})\big)<0.
\ee
Assume again by contradiction that this is not the case. Then, we deduce from~\eqref{utau*} that 
$$\sup_{\R^{N-1}\times[-A,A]}\big(\Delta(u-u_{\tau^*})-\lambda\,(u-u_{\tau^*})\big)=0$$ and there exists a sequence $(y_n)_{n\in\N}=(y'_n,y_{N,n})_{n\in\N}$ of points in $\R^{N-1}\times[-A,A]$ such that
$$\Delta(u-u_{\tau^*})(y_n)-\lambda\,(u(y_n)-u_{\tau^*}(y_n))\to0\ \hbox{ as }\ n\to+\infty.$$
Up to extraction of a subsequence, one can assume that $y_{N,n}\to y_{N,\infty}\in[-A,A]$ as $n\to+\infty$. Define
$$\tilde{U}_n(x)=\tilde{U}_n(x',x_N)=u(x'+y'_n,x_N)$$
in $\R^N$. As above, up to extraction of  a subsequence, the functions $\tilde{U}_n$ converge in $C^4_{loc}(\R^N)$ to a classical solution $\tilde{U}:\R^N\to[\alpha_-,\alpha_+]$ of~\eqref{fourth2} such that
$$\tilde{V}:=\tilde{U}-\tilde{U}_{\tau^*}=\tilde{U}-\tilde{U}(\cdot+\xi',\cdot-\tau^*)\ge 0\ \hbox{ in }\R^N,$$
$\tilde{W}:=\Delta\tilde{V}-\lambda\tilde{V}\le 0$ in $\R^N$ and $\tilde{W}(0,y_{N,\infty})=0$. Let $\tilde{\lambda}=\omega/\lambda=\beta-\lambda>0$ be the other root of $\tilde{\lambda}^2-\beta\lambda+\omega=0$. As in the proof of Lemma~\ref{lem3}, the function $\tilde{W}$ is a $C^2(\R^N)$ solution of
$$\Delta\tilde{W}-\tilde{\lambda}\,\tilde{W}=f(\tilde{U}(x))-f(\tilde{U}_{\tau^*}(x))+\omega\,\tilde{V}(x)=(\varsigma(x)+\omega)\,\tilde{V}(x),$$
where
$$\varsigma(x)=\frac{f(\tilde{U}(x))-f(\tilde{U}_{\tau^*}(x))}{\tilde{U}(x)-\tilde{U}_{\tau^*}(x)}\ \hbox{ if }\tilde{V}(x)=\tilde{U}(x)-\tilde{U}_{\tau^*}(x)\neq 0$$
and, say, $\varsigma(x)=0$ if $\tilde{V}(x)=\tilde{U}(x)-\tilde{U}_{\tau^*}(x)=0$. Since both functions $\tilde{U}$ and $\tilde{U}_{\tau^*}$ range in~$[\alpha_-,\alpha_+]$ and since $\omega>0$ satisfies~\eqref{hypomega}, one has $\varsigma(x)+\omega\ge0$ for all $x\in\R^N$. Recalling that $\tilde{V}\ge0$ in $\R^N$, one infers that $\Delta\tilde{W}-\tilde{\lambda}\,\tilde{W}\ge0$ in $\R^N$. Since $\tilde{W}$ is nonpositive in $\R^N$ and vanishes at the point $(0,y_{N,\infty})$, the strong maximum principle implies that $\tilde{W}=0$ in $\R^N$. In other words, $\Delta\tilde{V}-\lambda\tilde{V}=0$ in $\R^N$. But the function $\tilde{V}$ is nonnegative and bounded. By passing to the limit along a sequence $(z_n)_{n\in\N}$ such that $\tilde{V}(z_n)\to\sup_{\R^N}\tilde{V}$ as $n\to+\infty$, it follows immediately that $\sup_{\R^N}\tilde{V}\le0$. Hence $\tilde{V}=0$ in $\R^N$, that is, $\tilde{U}(x',x_N)=\tilde{U}(x'+\xi',x_N-\tau^*)$ for all $x=(x',x_N)\in\R^N$. The positivity of $\tau^*$ and the limits~\eqref{uniformalpha} lead to a contradiction. Finally,~\eqref{claim2} holds.\par
Now, since both properties~\eqref{claim1} and~\eqref{claim2} hold and since by interior elliptic estimates both functions $u$ an $\Delta u$ are uniformly continuous in $\R^N$, there is  a real number $\tau_*\in(0,\tau^*)$ such that
$$u_{\tau}\le u\ \hbox{ and }\ \Delta u_{\tau}-\lambda\,u_{\tau}\ge\Delta u-\lambda\,u\ \hbox{ in }\R^{N-1}\times[-A,A]\ \hbox{ for all }\tau\in[\tau_*,\tau^*].$$
For any $\tau\in[\tau_*,\tau^*]$, since $u\ge\alpha_+-\delta$ in $\R^{N-1}\times[A,+\infty)$ by~\eqref{choiceA}, it then follows from Lemma~\ref{lem3} applied to $z_1:=u_{\tau}(\cdot,\cdot+A)$ and $z_2=u(\cdot,\cdot+A)$ that
$$u_{\tau}\le u\ \hbox{ and }\ \Delta u_{\tau}-\lambda\,u_{\tau}\ge\Delta u-\lambda\,u\ \hbox{ in }\R^{N-1}\times[A,+\infty).$$
Similarly, for any $\tau\in[\tau_*,\tau^*]$, since $u_{\tau}=u(\cdot+\xi',\cdot-\tau)\le\alpha_-+\delta$ in $\R^{N-1}\times(-\infty,-A]$ by~\eqref{choiceA} and $\tau\ge\tau_*\ge0$, it then follows from Lemma~\ref{lem4} applied to $z_1:=u_{\tau}(\cdot,\cdot-A)$ and $z_2=u(\cdot,\cdot-A)$ that
$$u_{\tau}\le u\ \hbox{ and }\ \Delta u_{\tau}-\lambda\,u_{\tau}\ge\Delta u-\lambda\,u\ \hbox{ in }\R^{N-1}\times(-\infty,-A].$$
Putting together the previous results implies that $u_{\tau}\le u$ and $\Delta u_{\tau}-\lambda\,u_{\tau}\ge\Delta u-\lambda\,u$ in~$\R^N$ for all $\tau\in[\tau_*,\tau^*]$. Since $\tau_*\in(0,\tau^*)$, this contradicts the minimality of $\tau^*$ in~\eqref{deftau*} and, as a consequence, $\tau^*=0$.\par

\medbreak

{\it Step 3: conclusion.} Since $\tau^*=0$, it follows from~\eqref{deftau*} that
$$u_{\tau}\le u\ \text{ and }\ \Delta u_{\tau}-\lambda\,u_{\tau}\ge\Delta u-\lambda\,u \text{ in } \R^N$$
for all $\tau\ge0$. Furthermore, the arguments based on the strong second order elliptic maximum principle imply, as in Step~2, that, for all $\tau>0$, 
$$u_{\tau}<u \text{ and } \Delta u_{\tau}-\lambda\,u_{\tau}>\Delta u-\lambda\,u \text{ in }\R^N.$$ 
Choosing now $\xi'=0$ shows that $u$ is increasing in the variable $x_N$. Moreover, by choosing an arbitrary vector $\xi'\in\R^{N-1}$ and its opposite vector $-\xi'$, it follows that, for all $x=(x',x_N)\in\R^N$, $u(x'+\xi',x_N-\tau)<u(x',x_N)$ and $u(x'-\xi',x_N-\tau)<u(x',x_N)$ for all~$\tau>0$, whence $u(x'+\xi',x_N)\le u(x',x_N)$ and~$u(x'-\xi',x_N)\le u(x',x_N)$ by passing to the limit as $\tau\to0^+$. Since $x=(x',x_N)\in\R^N$ and $\xi'\in\R^{N-1}$ are arbitrary, this means that~$u(x'+\xi',x_N)=u(x',x_N)$ for all $\xi'\in\R^{N-1}$ and $(x',x_N)\in\R^N$. In other words, $u$ depends only on the variable $x_N$. The proof of Theo\-rem~\ref{th3} is thereby complete.
\end{proof}

We now prove Theorem~$\ref{th4}$. As for Theorem~\ref{th3}, we again rely on the sliding method, but this time we will slide the one-dimensional kink $\phi$ with respect to $u$.

\begin{proof}[Proof of Theorem~$\ref{th4}$]
Let $f$, $\delta\in(0,(\alpha_+-\alpha_-)/2)$, $\omega>0$, $\beta\ge2\sqrt{\omega}$, $u$ and $\phi$ be as in the statement of Theorem~\ref{th4}. Without loss of generality, let us assume that $u(x',x_N)\to\alpha_+$ as~$x_N\to\pm\infty$ uniformly in $x'\in\R^{N-1}$ and that $\inf_{\R^N}u>\alpha_-$. Remember from Corollary~\ref{coroneside} and Remark~\ref{remoneside} that this yields automatically
$$\alpha_-\le u\le u_+\ \hbox{ in }\R^N.$$
Even if it means decreasing $\delta>0$, one can then assume that $\alpha_-+\delta<\inf_{\R^N}u$. Let $\lambda>0$ be any root of the equation $\lambda^2-\beta\lambda+\omega=0$. Our goal is to show that $u=\alpha_+$ in $\R^N$.\par
Let $\xi$ be any given real number. As in the proof of Theorem~\ref{th4}, there is a real number~$A>0$ such that
\be\label{choiceA2}\left\{\baa{l}
u\ge\alpha_+-\delta\hbox{ in }\R^{N-1}\times[A,+\infty),\ \ \phi(\xi+\cdot)\le\alpha_-+\delta\hbox{ in }(-\infty,-A],\vspace{3pt}\\
\displaystyle|\Delta u|\le\frac{\lambda\,(\alpha_+-\alpha_--2\delta)}{2}\hbox{ in }\R^{N-1}\times(-\infty,-A]\,\cup\,\R^{N-1}\times[A,+\infty),\vspace{3pt}\\
\displaystyle|\phi''(\xi+\cdot)|\le\frac{\lambda\,(\alpha_+-\alpha_--2\delta)}{2}\hbox{ in }(-\infty,-A]\cup[A,+\infty).\eaa\right.
\ee
For any $\tau\in\R$, we define $\phi_{\tau}=\phi(\cdot+\xi-\tau)$. 

\medbreak

{\it Claim 1: for all $\tau\ge2A$, $\phi_{\tau}\le u$ and $\phi_{\tau}''-\lambda\,\phi_{\tau}\ge\Delta u-\lambda\,u$ in $\R^N$}.
Let $\tau\ge2A$. As in~\eqref{ineq1} and~\eqref{ineq2}, one has, for any $x'\in\R^{N-1}$,
$$\phi_{\tau}(A)=\phi(\xi+A-\tau)\le\alpha_-+\delta\le\alpha_+-\delta\le u(x',A)$$
and
$$\baa{rcl}
\phi_{\tau}''(A)-\lambda\,\phi_{\tau}(A)=\phi''(\xi+A-\tau)-\lambda\,\phi(\xi+A-\tau) & \!\!\!\ge\!\!\! & \displaystyle-\frac{\lambda\,(\alpha_+-\alpha_--2\delta)}{2}-\lambda\,(\alpha_-+\delta)\vspace{3pt}\\
& \!\!\!=\!\!\! & \displaystyle\frac{\lambda\,(\alpha_+-\alpha_--2\delta)}{2}-\lambda\,(\alpha_+-\delta)\vspace{3pt}\\
& \!\!\!\ge\!\!\! & \Delta u(x',A)-\lambda\,u(x',A).\eaa$$
Define $z_1,z_{2}:\R\to[\alpha_-,\alpha_+]$ by $z_{1}:=\phi_{\tau}(\cdot+A)=\phi(\cdot+\xi+A-\tau)$ and $z_2:=u(\cdot,\cdot+A)$. Since $z_{2}\ge\alpha_+-\delta$ in $\R^N_+$ and both $z_{1}$ and $z_{2}$ are classical solutions of~\eqref{fourth2} converging to $\alpha_+$ as $x_N\to+\infty$ uniformly in $x'\in\R^{N-1}$, it follows from the previous estimates and Lemma~\ref{lem3} that 
$$z_1\le z_2\ \text{ and }\ \Delta z_1-\lambda\,z_1\ge\Delta z_2-\lambda\,z_2 \text{ in } \R^N_+,$$ that is
$$\phi_{\tau}\le u\ \hbox{ and }\ \phi_{\tau}''-\lambda\,\phi_{\tau}\ge\Delta u-\lambda\,u\ \hbox{ in }\R^{N-1}\times[A,+\infty).$$
Similarly, since $z_1=\phi_{\tau}(\cdot+A)=\phi(\cdot+\xi+A-\tau)\le\alpha_-+\delta$ in $(-\infty,0]$ and both $z_1$ and~$z_2$ converge to $\alpha_-$ as $x_N\to-\infty$ uniformly in $x'\in\R^{N-1}$, it then follows from Lemma~\ref{lem4} that $z_1\le z_2$ and $\Delta z_1-\lambda\,z_1\ge\Delta z_2-\lambda\,z_2$ in $\R^N_-$, that is,
$$\phi_{\tau}\le u\ \hbox{ and }\ \phi_{\tau}''-\lambda\,\phi_{\tau}\ge\Delta u-\lambda\,u\ \hbox{ in }\R^{N-1}\times(-\infty,A].$$
Therefore the claim is proved.

\medbreak

Define now
\be\label{deftau*2}
\tau^*=\inf\big\{\tau>0,\ \phi_{\tau'}\le u\ \hbox{ and }\ \phi_{\tau'}''-\lambda\,\phi_{\tau'}\ge\Delta u-\lambda\,u\ \hbox{ in }\R^N\hbox{ for all }\tau'\ge\tau\big\}.
\ee
It follows from Claim~1 that $\tau^*$ is a nonnegative real number such that $\tau^*\le 2A$. Furthermore, by continuity, we infer that
\be\label{utau*bis}
\phi_{\tau^*}\le u\ \hbox{ and }\ \phi_{\tau^*}''-\lambda\,\phi_{\tau^*}\ge\Delta u-\lambda\,u\ \hbox{ in }\R^N.
\ee

\medbreak

{\it Claim 2: $\tau^*=0$}. Assuming by contradiction that $\tau^*>0$, we first aim to show that
\be\label{claim1bis}
\inf_{\R^{N-1}\times[-A,A]}(u-\phi_{\tau^*})>0.
\ee
Assume this inequality does not hold. Then $\inf_{\R^{N-1}\times[-A,A]}(u-\phi_{\tau^*})=0$ and there exists a sequence $(x_n)_{n\in\N}=(x'_n,x_{N,n})_{n\in\N}$ of points in $\R^{N-1}\times[-A,A]$ such that~$u(x_n)-~\phi_{\tau^*}(x_{N,n})\to~0$ as $n\to+\infty$. Up to extraction of a subsequence, one can assume that $x_{N,n}\to x_{N,\infty}\in[-A,A]$ as $n\to+\infty$. Define $U_n(x)=U_n(x',x_N)=u(x'+x'_n,x_N)$ in $\R^N$. As in the proof of Theorem~\ref{th3}, up to extraction of  a subsequence, the functions~$U_n$ converge in $C^4_{loc}(\R^N)$ to a classical solution $U:\R^N\to[\alpha_-,\alpha_+]$ of~\eqref{fourth2} such that~$V:=U-\phi_{\tau^*}\ge 0$ in $\R^N$, $V(0,x_{N,\infty})=0$ and $W:=\Delta V-\lambda\,V\le 0$ in $\R^N$. The strong maximum principle applied to $V$ yields $V=0$ in $\R^N$, that is, $U(x',x_N)=\phi(\xi+x_N-\tau^*)$ for all $x=(x',x_N)\in\R^N$. However, $U$ still satisfies~\eqref{uniformalpha2} (here $U(x',x_N)\to\alpha_+$ as $x_N\to\pm\infty$ uniformly in $x'\in\R^{N-1}$), while~$\phi(-\infty)=\alpha_-<\alpha_+$. One has then reached a contradiction, whence~\eqref{claim1bis} holds.\par
Similarly, we next show that
\be\label{claim2bis}
\sup_{\R^{N-1}\times[-A,A]}\big(\Delta u-\phi''_{\tau^*}-\lambda\,(u-\phi_{\tau^*})\big)<0.
\ee
Assume again by contradiction that this inequality does not hold. Then, by~\eqref{utau*bis}, there is a sequence $(y_n)_{n\in\N}=(y'_n,y_{N,n})_{n\in\N}$ of points in~$\R^{N-1}\times[-A,A]$ such that $\Delta u(y_n)-\phi''_{\tau^*}(y_{N,n})-\lambda\,(u(y_n)-\phi_{\tau^*}(y_{N,n}))\to0$ as $n\to+\infty$. Define {$\tilde{U}_n(x)=\tilde{U}_n(x',x_N)=u(x'+y'_n,x_N)$} in $\R^N$. Up to extraction of  a subsequence, one can assume that $y_{N,n}\to y_{N,\infty}\in[-A,A]$ as $n\to+\infty$ and that the functions $\tilde{U}_n$ converge in $C^4_{loc}(\R^N)$ to a classical solution~$\tilde{U}:\R^N\to[\alpha_-,\alpha_+]$ of~\eqref{fourth2} such that $\tilde{V}:=\tilde{U}-\phi_{\tau^*}\ge 0$ in $\R^N$, $\tilde{W}:=\Delta\tilde{V}-\lambda\tilde{V}\le 0$ in $\R^N$ and $\tilde{W}(0,y_{N,\infty})=0$. Let $\tilde{\lambda}=\omega/\lambda=\beta-\lambda>0$ be the other root of $\tilde{\lambda}^2-\beta\tilde{\lambda}+\omega=0$. As in the proof of Lemma~\ref{lem3}, the function $\tilde{W}$ is a $C^2(\R^N)$ solution of
$$\Delta\tilde{W}-\tilde{\lambda}\,\tilde{W}=(\varsigma(x)+\omega)\tilde{V},$$
where
$$\varsigma(x)=\frac{f(\tilde{U}(x))-f(\phi_{\tau^*}(x_N))}{\tilde{U}(x)-\phi_{\tau^*}(x_N)}\ \hbox{ if }\tilde{V}(x)=\tilde{U}(x)-\phi_{\tau^*}(x_N)\neq 0$$
and, say, $\varsigma(x)=0$ if $\tilde{V}(x)=\tilde{U}(x)-\phi_{\tau^*}(x_N)=0$. Since both functions $\tilde{U}$ and $\phi_{\tau^*}$ range in $[\alpha_-,\alpha_+]$ and since $\omega>0$ satisfies~\eqref{hypomega}, one has $\varsigma(x)+\omega\ge0$ for all $x\in\R^N$. Recalling that $\tilde{V}\ge0$ in $\R^N$, one infers that $\Delta\tilde{W}-\tilde{\lambda}\,\tilde{W}\ge0$ in $\R^N$. Since $\tilde{W}$ is nonpositive in $\R^N$ and vanishes at the point $(0,y_{N,\infty})$, the strong maximum principle implies that $\tilde{W}=0$ in $\R^N$. In other words, $\Delta\tilde{V}-\lambda\tilde{V}=0$ in $\R^N$. But the function $\tilde{V}$ is nonnegative and bounded. It then follows that $\tilde{V}=0$ in $\R^N$, that is, $\tilde{U}(x',x_N)=\phi(\xi+x_N-\tau^*)$ for all~$x=(x',x_N)\in\R^N$. The limits~\eqref{uniformalpha2} lead again to a contradiction, whence~\eqref{claim2bis} holds.\par
We can now conclude Claim 2. From~\eqref{claim1bis},~\eqref{claim2bis} and the uniform continuity of $u$, $\Delta u$, $\phi$ and $\phi''$, there is $\tau_*\in(0,\tau^*)$ such that
$$\phi_{\tau}\le u\ \hbox{ and }\ \phi_{\tau}''-\lambda\,\phi_{\tau}\ge\Delta u-\lambda\,u\ \hbox{ in }\R^{N-1}\times[-A,A]\ \hbox{ for all }\tau\in[\tau_*,\tau^*].$$
For any $\tau\in[\tau_*,\tau^*]$, since $u\ge\alpha_+-\delta$ in $\R^{N-1}\times[A,+\infty)$ by~\eqref{choiceA2}, it then follows from Lemma~\ref{lem3} applied to $z_1=\phi_{\tau}(\cdot+A)=\phi(\cdot+\xi+A-\tau)$ and $z_2=u(\cdot,\cdot+A)$ that
$$\phi_{\tau}\le u\ \hbox{ and }\ \phi_{\tau}''-\lambda\,\phi_{\tau}\ge\Delta u-\lambda\,u\ \hbox{ in }\R^{N-1}\times[A,+\infty).$$
Similarly, for any $\tau\in[\tau_*,\tau^*]$, since $\phi_{\tau}=\phi(\cdot+\xi-\tau)\le\alpha_-+\delta$ in $(-\infty,-A]$ by~\eqref{choiceA2} and $\tau\ge\tau_*\ge0$, it then follows from Lemma~\ref{lem4} applied to $z_1:=\phi_{\tau}(\cdot-A)$ and $z_2=u(\cdot,\cdot-A)$ that
$$\phi_{\tau}\le u\ \hbox{ and }\ \phi''_{\tau}-\lambda\,\phi_{\tau}\ge\Delta u-\lambda\,u\ \hbox{ in }\R^{N-1}\times(-\infty,-A].$$
Putting together the previous results implies that $\phi_{\tau}\le u$ and $\phi''_{\tau}-\lambda\,\phi_{\tau}\ge\Delta u-\lambda\,u$ in $\R^N$ for all $\tau\in[\tau_*,\tau^*]$. Since $\tau_*\in(0,\tau^*)$, this contradicts the minimality of $\tau^*$ in~\eqref{deftau*2} and thus implies $\tau^*=0$.\par

\medbreak

{\it Conclusion}. We have shown that $\phi_0\le u$ in $\R^N$, that is, $\phi(\xi+x_N)\le u(x',x_N)$ for all $(x',x_N)\in\R^N$. Since $\xi$ was an arbitrary real number, it follows by passing to the limit~$\xi\to+\infty$ that~$\alpha_+\le u(x',x_N)$ for all $(x',x_N)\in\R^N$. But $u$ was assumed to range in $[\alpha_-,\alpha_+]$. In other words, $u=\alpha_+$ in $\R^N$ and the proof of Theorem~\ref{th4} is thereby complete. 
\end{proof}


\end{document}